\documentclass[11pt]{article}
\usepackage{tocbibind} 
\usepackage{amsmath,amsthm,amssymb}
\usepackage{fancyhdr}
\usepackage[british]{babel}
\usepackage{geometry,mathtools}
\usepackage{enumitem}
\usepackage{algpseudocode}
\usepackage{dsfont}
\usepackage{centernot}
\usepackage{xstring}
\usepackage{colortbl}

\usepackage{cancel}
\usepackage[numbers,sort]{natbib}
\usepackage[section]{algorithm}
\usepackage{graphicx}
\usepackage[font={footnotesize}]{caption}
\usepackage[usenames,dvipsnames,table]{xcolor}
\usepackage{pstricks,tikz}
\usetikzlibrary{patterns,calc,fadings,decorations.pathreplacing}
\usepackage[h]{esvect}
\usepackage[
bookmarksopen=true,
bookmarksopenlevel=1,
colorlinks=true,
linkcolor=darkblue,
linktoc=page,
citecolor=darkblue,
urlcolor=darkblue
]{hyperref}
\usepackage{bbm}

\numberwithin{equation}{section}

\definecolor{darkblue}{rgb}{0,0,0.5}

\newdimen\margin
\def\textno#1&#2\par{
   \margin=\hsize
   \advance\margin by -4\parindent
          \setbox1=\hbox{\sl#1}
   \ifdim\wd1 < \margin
      $$\box1\eqno#2$$
   \else
      \bigbreak
      \hbox to \hsize{\indent$\vcenter{\advance\hsize by -3\parindent
      \it\noindent#1}\hfil#2$}
      \bigbreak
   \fi}

\newtheorem{theorem}[algorithm]{Theorem}
\newtheorem{prop}[algorithm]{Proposition}
\newtheorem{lemma}[algorithm]{Lemma}
\newtheorem{cor}[algorithm]{Corollary}
\newtheorem{fact}[algorithm]{Fact}

\theoremstyle{definition}

\def\lateproof#1{\removelastskip\penalty55\medskip\noindent\begin{stepenv}\end{stepenv}{\bf Proof of #1. }} 

\def\noproof{{\unskip\nobreak\hfill\penalty50\hskip2em\hbox{}\nobreak\hfill%
       $\square$\parfillskip=0pt\finalhyphendemerits=0\par}\goodbreak}
\def\endproof{\noproof\bigskip}

\newcounter{stepenv}
\newenvironment{stepenv}[1][]{\refstepcounter{stepenv}}{}

\newcounter{step}[stepenv]

\newcounter{substep}[step]
\renewcommand{\thesubstep}{\thestep.\arabic{substep}}

\newcounter{claim}[stepenv]
\newenvironment{claim}[1][]{\refstepcounter{claim}\par\medskip\noindent%
        \textit{Claim~\theclaim. #1} \itshape\rmfamily}{\medskip}

\newcommand{\cB}{\mathcal{B}}
\newcommand{\cC}{\mathcal{C}}

\newcommand{\cH}{\mathcal{H}}

\newcommand{\cL}{\mathcal{L}}
\newcommand{\cM}{\mathcal{M}}

\newcommand{\cP}{\mathcal{P}}

\newcommand{\cR}{\mathcal{R}}
\newcommand{\cS}{\mathcal{S}}

\newcommand{\cW}{\mathcal{W}}

\newcommand{\cY}{\mathcal{Y}}
\newcommand{\cZ}{\mathcal{Z}}

\newcommand{\bN}{\mathbb{N}}

\newcommand{\EE}{\mathbb{E}}
\newcommand{\PP}{\mathbb{P}}

\def\eps{{\varepsilon}}

\newcommand{\wx}{\mathbf{x}}

\newcommand{\ordered}[1]{\overrightarrow{#1}}

\newcommand{\defn}{\emph}

\newcommand{\prob}[1]{\mathrm{\mathbb{P}}\left[#1\right]}

\newcommand{\Set}[1]{\{#1\}}

\def\In{\subseteq}

\newcommand{\IND}{\mathbbm{1}}

\def\Pr{\mathbf{Pr}}

\def\COMMENT#1{}
\def\TASK#1{}
\let\TASK=\footnote             
\let\COMMENT=\footnote          

\begin{document}

\title{Tight Hamilton cycles with high discrepancy}

\author{
Lior Gishboliner \thanks{Department of Mathematics, University of Toronto, ON, Canada.
\\ \emph{Email}: \texttt{lior.gishboliner@utoronto.ca}. When conducting this work, LG was supported by SNSF grant 200021\_196965.} 
\and
Stefan Glock \thanks{Fakultät für Informatik und Mathematik, Universität Passau, Germany.
\\ \emph{Email}: \texttt{stefan.glock@uni-passau.de,amedeo.sgueglia@uni-passau.de}.
This research was conducted while AS was affiliated with University College London and supported by the Royal Society.}
\and
Amedeo Sgueglia \footnotemark[2]
}

\date{}

\maketitle

\begin{abstract} 
In this paper, we study discrepancy questions for spanning subgraphs of $k$-uniform hypergraphs. Our main result is that, for any integers $k \ge 3$ and $r \ge 2$, any $r$-colouring of the edges of a $k$-uniform $n$-vertex hypergraph $G$ with minimum $(k-1)$-degree $\delta(G) \ge (1/2+o(1))n$ contains a tight Hamilton cycle with high discrepancy, that is, with at least $n/r+\Omega(n)$ edges of one colour. The minimum degree condition is asymptotically best possible and our theorem also implies a corresponding result for perfect matchings.
Our tools combine various structural techniques such as Tur\'an-type problems and hypergraph shadows with probabilistic techniques such as random walks and the nibble method. We also propose several intriguing problems for future research.
\end{abstract}

\section{Introduction}

In discrepancy theory, the basic question is whether a structure can be partitioned in a balanced way, or if there is always some ``discrepancy'' no matter how the partition is made. Formally, 
let $\mathcal{H}$ be a hypergraph and let 
$f : V(\mathcal{H}) \rightarrow 
\{\text{red, blue}\}$ be a 2-colouring of its vertices. 
For an edge $e \in E(\mathcal{H})$ and a colour $c$, let $c(e) := \{x \in e : f(x) = c\}$. 
The discrepancy of $e$ is defined as 
$D_f(e) := \big||\text{red}(e)| - |\text{blue}(e)|\big| = 2 \cdot \max_{c \in \{\text{red, blue}\}} \left(
|c(e)| - \frac{|e|}{2} \right)$; the larger $D_f(e)$ is, the less balanced is the colouring of $e$. 
The discrepancy of $\mathcal{H}$ is then defined as $\min_f \max_e D_f(e)$. In other words, the discrepancy measures the maximum imbalance that is guaranteed to occur in every $2$-colouring of $V(\mathcal{H})$. 
Discrepancy of hypergraphs is a classical topic in combinatorics; we refer the reader to~\cite[Chapter 13]{AS:08} for an introduction.
The notion of discrepancy naturally generalizes to more than $2$ colours: 
For a hypergraph $\mathcal{H}$, an $r$-colouring $f : V(\mathcal{H}) \rightarrow [r]$ and an edge $e \in E(\mathcal{H})$, define the discrepancy of $e$ as 
$D_f(e) := r \cdot \nolinebreak \max_{c \in [r]} \left(
|c(e)| - \frac{|e|}{r} \right)$. 
This coincides with the above definition of $D_f(e)$ for the case $r=2$. The $r$-colour discrepancy of $\mathcal{H}$ is then defined as $\min_f \max_e D_f(e)$. 

There are many works studying discrepancy problems for hypergraphs arising from graphs, namely, when $V(\mathcal{H})$ is the {\em edge set} of a graph $G$ and $E(\mathcal{H})$ is a family of subgraphs of $G$. Two early results of this type are the theorem of 
Erd\H{o}s and Spencer~\cite{ES:72} on the discrepancy of cliques in the complete graph, and the work of Erd\H{o}s, F\"uredi, Loebl and S\'os~\cite{EFLS:95} on the discrepancy of copies of a given spanning tree in the complete graph. In recent years there has been a lot of interest in discrepancy problems in general graphs, and there are by now many works studying conditions that guarantee the existence of high-discrepancy subgraphs of various types, such as perfect matchings and Hamilton cycles~\cite{BCJP:20, FHLT:21,GKM_Hamilton, GKM_trees}, spanning trees~\cite{GKM_trees}, $H$-factors~\cite{BCPT:21,BCG:23} and powers of Hamilton cycles~\cite{Bradac:22}. See also~\cite{FL:22,GKM_oriented} for an oriented analogue. 
Many of the results study minimum degree thresholds for linear discrepancy, namely, they determine how large the minimum degree of $G$ should (asymptotically) be so that $G$ is guaranteed to contain a subgraph of a certain type with discrepancy $\Omega(n)$.
For example, in~\cite{BCJP:20, FHLT:21, GKM_trees}, it is shown that for every $\varepsilon > 0$, every $r$-edge-colouring of an $n$-vertex graph with minimum degree $(\frac{r+1}{2r} + \varepsilon)n$ has a Hamilton cycle and a perfect matching with linear discrepancy, and the constant $\frac{r+1}{2r}$ is best possible.
We will often say {\em high discrepancy} to mean linear discrepancy, i.e., discrepancy $\Omega(n)$.

In this work, we study discrepancy problems in $k$-uniform hypergraphs (short: \defn{$k$-graphs}), 
in analogy to the aforementioned works for graphs. In this context, it is worth mentioning the seminal result of Alon, Frankl and Lov\'asz~\cite{AFL:86} in which they proved, using topological methods, the following conjecture of Erd\H{o}s: If the edges of the complete $n$-vertex $k$-graph $K^{(k)}_n$ are coloured with $r$ colours, and $n\ge (r-1)(s-1)+sk$, then there exists a monochromatic matching of size~$s$. This generalizes Kneser's conjecture which corresponds to the case $k=2$ and was resolved by Lov\'asz~\cite{Lovasz:78}. The Alon--Frankl--Lov\'asz result implies in particular that any $2$-edge-colouring of $K^{(k)}_n$ contains a monochromatic matching of size at least $\lfloor \frac{n}{k+1} \rfloor$ and, by arbitrarily adding edges, this can be extended to a perfect matching with high discrepancy (assuming $k\mid n$ of course).

Our main result is the determination of the minimum $(k-1)$-degree threshold for the discrepancy of perfect matchings and tight Hamilton cycles in $k$-graphs, thereby establishing a discrepancy version of the celebrated theorem of R\"odl, Ruci\'nski and Szemer\'edi~\cite{RRS:08}. 
Recall that a \defn{tight Hamilton cycle} of a $k$-graph $G$ is a cyclic ordering $v_1,\dots,v_n$ of the vertices of $G$ such that $v_iv_{i+1}\dots v_{i+k-1}$ is an edge for every $1 \leq i \leq n$, with indices taken modulo $n$. Before stating our main result, let us give some background on perfect matchings and Hamilton cycles in hypergraphs of large minimum degree. 
For a $k$-graph $G$ and a set $S \subseteq V(G)$, we say that the \emph{degree} of $S$ in $G$, denoted by $d_G(S)$, is the number of edges containing~$S$.
We use $\delta(G)$ to denote the \emph{minimum $(k-1)$-degree}, which is the minimum of $d_G(S)$ over all $(k-1)$-sets $S \subseteq V(G)$. 
In their seminal paper which introduced the absorbing method systematically, R\"odl, Ruci\'nski and Szemer\'edi~\cite{RRS:08} showed that for every $\varepsilon > 0$, any $n$-vertex $k$-graph $G$ with $\delta(G)\ge (1/2+\varepsilon)n$ contains a tight Hamilton cycle. 
Moreover, this is best possible, as there are $k$-graphs $G$ with $\delta(G) = n/2 - O(1)$ and no tight Hamilton cycle (see~\cite[Theorem $3$]{KK:99}).

Let us now consider discrepancy of tight Hamilton cycles. 
Mansilla Brito~\cite{mansilla} showed that if a $3$-graph $G$ satisfies $\delta(G) \ge (5/6+\eps)n$, then it contains a tight Hamilton cycle with high discrepancy.
We improve this to $\delta(G) \geq (1/2+\eps)n$ and show that this holds for any uniformity $k \geq 3$ and any (fixed) number of colours $r \geq 2$. This result is best possible since $\delta(G) \geq \frac{n}{2} - O(1)$ is needed even to guarantee the existence of a tight Hamilton cycle.
Thus, for $k \geq 3$, the threshold for the discrepancy of Hamilton cycles is the same as the existence threshold, and does not depend on the number of colours $r$. This is in contrast to the graph case, where the discrepancy threshold is strictly larger than the existence threshold and decreases as $r$ increases (see~\cite{BCJP:20,GKM_trees,FHLT:21}).

\begin{theorem}
\label{thm:main}
For all $k,r\in \bN$ with $k\ge 3$ and $r\ge 2$, and all $\eps>0$, there exists $\mu>0$ such that the following holds for all sufficiently large $n$.
Let $G$ be an $n$-vertex $k$-graph with $\delta(G)\ge (1/2+\eps)n$ whose edges are $r$-coloured. Then there exists a tight Hamilton cycle in $G$ which contains at least $(1+\mu)\frac{n}{r}$ edges of the same colour.
\end{theorem}

Note that if $n$ is divisible by $k$, then a tight Hamilton cycle decomposes into $k$ perfect matchings. So by using Theorem~\ref{thm:main} and averaging, we obtain the following corollary, which was proved independently by Balogh, Treglown and Z\'arate-Guer\'en~\cite{BTZ-G:24}.

\begin{cor}\label{cor:PM}
For all $k,r\in \bN$ with $k\ge 3$ and $r\ge 2$, and all $\eps>0$, there exists $\mu>0$ such that the following holds for all sufficiently large $n$ divisible by $k$.
Let $G$ be an $n$-vertex $k$-graph with $\delta(G)\ge (1/2+\eps)n$ whose edges are $r$-coloured. Then there exists a perfect matching in $G$ which contains at least $(1+\mu)\frac{n}{rk}$ edges of the same colour.
\end{cor}

The constant $1/2$ in Corollary~\ref{cor:PM} is tight, as there exist $n$-vertex $k$-graphs with $\delta(G)=n/2-O(1)$ and no perfect matching (see~\cite{KO:06matchings,RRS:09}). 

Similarly, Theorem~\ref{thm:main} also implies an upper bound of $1/2$ for the discrepancy threshold of Hamilton $\ell$-cycles, where, for $1 \le \ell \le k-1$ and for $n$ divisible by $k-\ell$, a {\em Hamilton $\ell$-cycle} on $n$ vertices
is a cyclic order $v_1,\dots,v_n$ of the vertices such that 
$v_iv_{i+1}\dots v_{i+k-1}$ is an edge for every $i$ divisible by $k-\ell$ (so any two such consecutive edges intersect in exactly $\ell$ vertices).
Note that the case $\ell=k-1$ corresponds to a tight Hamilton cycle. 
Observe indeed that if $n$ is divisible by $k-\ell$, then a tight Hamilton cycle on $n$ vertices decomposes into $k-\ell$ Hamilton $\ell$-cycles. Thus, by using Theorem~\ref{thm:main} and averaging, we get that in every $r$-edge-colouring of a $k$-graph $G$ on $n$ vertices, with $n$ divisible by $k-\ell$ and $\delta(G) \geq (1/2+\varepsilon)n$, there is a Hamilton $\ell$-cycle with at least $(1+\mu)\frac{n}{r(k-\ell)}$ edges of the same colour. Unlike Theorem~\ref{thm:main} and Corollary~\ref{cor:PM}, here we do not know whether the constant $\frac{1}{2}$ is tight, i.e., whether minimum degree $\frac{n}{2}$ is necessary. See the concluding remarks for more on this.

\bigskip
\textbf{Organization of the paper.} 
In Section~\ref{sec:sketch}, we provide an overview of the key steps of our proof. In Section~\ref{sec:preliminaries}, we summerise some known tools. Section~\ref{sec:key lemma} contains our key structural lemma which is already sufficient to prove the case of perfect matchings, i.e., Corollary~\ref{cor:PM}. In Sections~\ref{sec:fractional}--\ref{sec:main proof}, we use additional methods to deal with Hamilton cycles.
In the final section, we collect various other problems concerning discrepancy of spanning structures in hypergraphs which seem very interesting for further research.

\bigskip
\textbf{Notation.} 
For a set $V$ and a natural number $m$, we write $\binom{V}{m}$ to denote the set of all $m$-subsets of $V$. 
We write $(V)_m$ to denote the set of all ordered $m$-tuples of distinct elements of $V$.
We use capital letters with arrows above to denote ordered tuples $\ordered{S} \in (V)_m$.
We shall subsequently drop the arrow to denote the unordered version of this $m$-tuple, so that if $\ordered{S}:=(v_1,v_2,\dots,v_m)$, then $S$ denotes the set $\{v_1,\dots,v_m\}$.
Moreover, we write $\overleftarrow{S}$ to denote the ordered $m$-tuple obtained by reversing the ordering of $\ordered{S}$, so that $\overleftarrow{S}:=(v_m,v_{m-1},\dots,v_1)$.

Let $G$ be a $k$-graph.
For $v  \in V(G)$ and a $(k-1)$-set $S \subseteq V(G)$, we say that $v$ is a \emph{neighbour} of $S$ in $G$ if $S \cup \{ v\}$ is an edge in $G$, and we denote the set of neighbours of $S$ in $G$ by $N_G(S)$.
The \emph{shadow} of $G$ is the $(k-1)$-graph on $V(G)$ whose edges are the $(k-1)$-sets which are contained in at least one edge of $G$.

Given a tight path $P = v_1 v_2 \dots v_\ell$ on $\ell\ge k-1$ vertices in a $k$-graph, we say that $P$ \defn{connects} the ordered $(k-1)$-sets $(v_1, v_2, \dots, v_{k-1})$ and $(v_\ell, v_{\ell-1}, \dots, v_{\ell-k+2})$, which we call the \emph{ends} of~$P$.
The choice of taking $(v_\ell, v_{\ell-1}, \dots, v_{\ell-k+2})$ rather than $(v_{\ell-k+2}, v_{\ell-k+3}, \dots, v_{\ell})$ as an end of $P$ is intentional and due to the fact that $P$ is an undirected path. We call $\ell$ the \emph{order} of $P$.

Given a $2$-edge-colouring of $G$ where we allow edges to receive multiple colours, we call the edges receiving both colours \emph{double-coloured}.

For $a$, $b$, $c \in (0, 1]$, we write
$a \ll b \ll c$ in our statements to mean that there are increasing functions $f, g : (0, 1] \to (0, 1]$ such that whenever $a \le f (b)$ and $b \le g(c)$, then the subsequent result holds.
Moreover, when using the Landau symbols $O(\cdot), \Omega(\cdot )$, subscripts denote variables that the implicit constant may depend on. 

We say that an event holds \defn{with high probability (w.h.p.)} if the probability that it holds tends to $1$ as the number of vertices $n$ tends to infinity.

\section{Proof overview}\label{sec:sketch}

Throughout this section, we let $G$ be an $r$-edge-coloured $n$-vertex $k$-graph with $\delta(G) \ge (1/2+ \nolinebreak \eps)n$.
We will first sketch a proof that $G$ contains a perfect matching with high discrepancy. Subsequently we will discuss what more needs to be done in order to find a tight Hamilton cycle with high discrepancy.

\medskip
{\bf Perfect matchings.} We start by assuming that $r=2$; we will handle the case of an arbitrary number of colours later on.

We aim to find an (edge-coloured) ``gadget'' in $G$ with the property that it contains a perfect matching where the majority colour is red and a perfect matching where the majority colour is blue.
Such a gadget can then be used to ``push'' the majority colour.
A natural candidate is the alternating $k$-grid, which is defined as follows.
A \emph{$k$-grid} is the $k$-graph on vertices $\{x_{ij}: 1 \le i,j \le k\}$ and with edges $x_{i1} \cdots x_{ik}$ for each $i \in [k]$ (which we will call \emph{horizontal} edges) and $x_{1j} \cdots x_{kj}$ for each $j \in [k]$ (which we will call \emph{vertical} edges).
An \emph{alternating $k$-grid} is a $2$-edge-coloured $k$-grid with all horizontal edges red and all vertical edges blue (cf.~Figure~\ref{fig_gadget}).
Observe that the horizontal edges form a red perfect matching and the vertical edges form a blue perfect matching.
If we can find linearly many such vertex-disjoint gadgets, say $\eps n/2k^2$ many, then, after removing them, the resulting $k$-graph $G'$ still has $\delta(G')\ge (1/2+\eps/2)|V(G')|$ and hence has a perfect matching~$M'$ (cf.~\cite{KO:06matchings,RRS:09}). Without loss of generality, assume that at least half of the edges of $M'$ are red. Then, for each of the gadgets, take the red perfect matching.
The union of all such matchings gives a perfect matching of $G$ with at least $(1-\frac{\eps}{2})\frac{n}{2k}+\frac{\eps}{2k}n=(1+\frac{\eps}{2})\frac{n}{2k}$ red edges, thus providing a perfect matching with high discrepancy.

\begin{figure}[htpb]
	\begin{center}
		\includegraphics{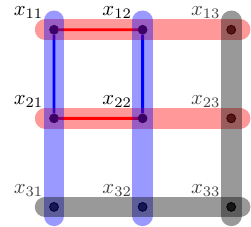}
		\captionof{figure}{An alternating $2$-grid on vertices $\{x_{11},x_{12},x_{21},x_{22}\}$ and a near-alternating $3$-grid on vertices $\{x_{11},x_{12},x_{13},x_{21},x_{22},x_{23},x_{31},x_{32},x_{33}\}$.
			The grey edges stand for edges whose colour is arbitrary.
		}
		\label{fig_gadget}
	\end{center}
\end{figure}

Unfortunately, an alternating $k$-grid does not necessarily exist in $G$, not even if $G$ is complete and has many edges of both colours. (For instance, if we choose a subset $A\In V(G)$ and colour all edges which intersect $A$ blue and all edges which do not intersect $A$ red, then it is easily verified that there is no alternating $k$-grid.)
However, our key lemma says that, unless the colouring is almost monochromatic, we can guarantee a {\em near-alternating $k$-grid}, namely, a $2$-edge-coloured $k$-grid such that all horizontal edges but at most one are red, and all vertical edges but at most one are blue (cf.~Figure~\ref{fig_gadget}).

\begin{enumerate}[label=(\textbf{L1})]
\item \label{overview:L1} If $\delta(G)\ge (1/2+\varepsilon)n$, then either $G$ contains a near-alternating $k$-grid, or $G$ is almost monochromatic.
\end{enumerate}

The formal statement of~\ref{overview:L1} is offered by Lemma~\ref{lem:key lemma}. We defer a proof sketch of this result to Section~\ref{sec:key lemma}.
Applying~\ref{overview:L1} repeatedly gives that either $G$ contains linearly many vertex-disjoint near-alternating $k$-grids, or its colouring is almost monochromatic.
In the first case, one can apply the same argument as above, with alternating $k$-grids replaced by near-alternating $k$-grids: this still works since, for each gadget, we can decide to cover its $k^2$ vertices with a perfect matching containing more red edges or more blue edges (as $k \geq 3$).
In the second case, when almost all edges of $G$ have the same colour, one can use standard methods to show that $G$ contains an almost-monochromatic perfect matching.
We remark this here for the benefit of the reader, but will actually not implement these steps since the result for perfect matchings will follow from the more general theorem about tight Hamilton cycles.

The case of an arbitrary number $r$ of colours can be handled by identifying the first $r-1$ colours into one colour (say ``blue") and applying the 2-colour argument outlined above. 
The key point is that it suffices to find a perfect matching of $G$ that either has at least $\frac{n}{rk}+\Omega(n)$ edges in the $r$th colour or at least $\frac{(r-1)n}{rk}+\Omega(n)$ edges in blue, because in the latter case, averaging gives that one of the first $r-1$ colours appears at least $\frac{n}{rk}+\Omega(n)$ times.
Such a perfect matching can be found using the same strategy as above, consisting of first removing near-alternating grids and then covering the rest with a perfect matching $M'$, and by applying a ``biased'' case distinction for~$M'$. Namely, in $M'$, either a $(1-\frac{1}{r})$-fraction of edges is blue or a $\frac{1}{r}$-fraction has the $r$th colour, and whichever case holds, we can use the gadgets to ``push'' the relevant colour(s).

\medskip
{\bf Tight Hamilton cycles.} We now discuss how to find a tight Hamilton cycle with high discrepancy.
For hypergraphs with minimum degree above $1/2$ (to which we refer informally as \emph{Dirac hypergraphs}), there are well-known tools which allow one to connect a given set of disjoint tight paths into a tight Hamilton cycle (see Section~\ref{sec:dirac_graphs}).
Therefore, if these paths are of high discrepancy and only ``miss'' a small number of edges to close a Hamilton cycle, then we are done as, no matter which colours are used in the completion, the discrepancy cannot be ruined anymore.
A crucial step of our argument is to find such paths (the formal statement is offered by Lemma~\ref{lem:path_forest}).
\begin{enumerate}[label=(\textbf{L2})]
    \item \label{overview:L2} If $\delta(G) \ge (1/2+\varepsilon)n$, then $G$ contains a collection of vertex-disjoint tight paths, whose union contains $(1-o(1))n$ edges and has high discrepancy.
\end{enumerate}

In order to show the above, we proceed as follows.
We use our key lemma to find a {\em perfect fractional matching} $\wx$ such that $\wx$ is ``normal'', i.e.~each edge has weight $\Theta(n^{-k+1})$, and such that $\wx$ has high discrepancy, in the sense that the total weight received by some colour class (say ``red") is significantly above the average, i.e.~larger than $n/(rk) + \Omega(n)$ (see Lemma~\ref{lem:normal PFM_discrepancy}, and see Section~\ref{sec:dirac_graphs} for the definition of a perfect fractional matching).
We then use $\wx$ to define a random walk $\cY$ on $V(G)$, such that a path of order $t$ sampled according to the first $t$ vertices of $\cY$ (conditioning on being self avoiding) has the following properties: 
Every vertex is approximately equally likely to be contained in the path and the probability of an edge $e$ appearing in the path is roughly proportional to $\wx(e)$ (see Lemma~\ref{lem:sampling}).

We sample $N$ paths of order $t$ independently, where $t$ is a sufficiently large constant and $N:=n^{t-1/2}$. 
Finally, we define an auxiliary $t$-uniform hypergraph $\cH$ with vertex set $V(G)$ and edges corresponding to the vertex sets of the sampled paths. (The choice of $N$ ensures that this hypergraph is rather dense, but still we do not expect too many parallel edges.)
Owing to the first property of $\cY$ mentioned above, the $t$-graph $\cH$ is almost-regular and, obviously, its maximum $2$-degree is at most $n^{t-2}$.
By a fundamental theorem of Frankl, Rödl~\cite{FR:85} and Pippenger (see~\cite{PS:89}), we can then establish that $\cH$ contains an almost-perfect matching, which corresponds to a collection of vertex-disjoint tight paths of order $t$ covering almost all the vertices. However, our goal is to find such a collection with high discrepancy. Owing to the second property of $\cY$ mentioned above and the fact that the weight of red edges is significantly above the average, the paths in $G$ which correspond to the edges in $\cH$ are likely to contain many red edges. A result concerning finding hypergraph matchings with pseudorandom properties due to Ehard, Glock and Joos~\cite{EGJ:20} (see Theorem~\ref{thm:EGJ}) then allows us to find an almost-perfect matching in $\cH$ such that the corresponding collection of vertex-disjoint paths indeed has high discrepancy. This proves~\ref{overview:L2} and establishes Theorem~\ref{thm:main}.

\section{Preliminaries}\label{sec:preliminaries}
In this section, we collect some preliminary results.

\subsection{Dirac hypergraphs}
\label{sec:dirac_graphs}
In this subsection we state three results concerning hypergraphs with minimum $(k-1)$-degree above $n/2$.
The first is the so-called ``Connecting Lemma'' which allows to connect any two disjoint ordered $(k-1)$-tuples of vertices by a tight path of bounded length.

\begin{lemma}[Lemma 2.4 in~\cite{RRS:08}]\label{lem:path connect}
Let $1/n \ll \eps \ll 1/k$ with $k \in \mathbb{N}$ and $k \ge 3$. 
Let $G$ be a $k$-graph on $n$ vertices with $\delta(G)\ge (1/2+\eps)n$. Then for every two disjoint ordered $(k-1)$-subsets of vertices
of $G$, there is a tight path in $G$ of order at most $2k/\eps^2$ which connects them. 
\end{lemma}

The second result is the so-called ``Absorbing Lemma'' and shows the existence of an absorber for tight paths in Dirac hypergraphs.
\begin{lemma}[Lemma 2.1 in~\cite{RRS:08}]\label{lem:absorber}
    Let $1/n \ll \beta \ll \mu \ll \eps, 1/k$, with $k \in \mathbb{N}$ and $k \ge 3$.
    Let $G$ be a $k$-graph on $n$ vertices with $\delta(G) \ge (1/2+\eps)n$.
    Then there exists a tight path $P$ with $|V(P)| \le \mu n$ such that for every subset $W \subseteq V(G) \setminus V(P)$ of size $|W| \le \beta n$, there is a tight path $\tilde{P}$ in $G$ with $V(\tilde{P})=V(P) \cup W$ and such that $\tilde{P}$ has the same ends as $P$.
\end{lemma}

The third result concerns the existence of ``balanced" perfect fractional matchings. Let us introduce the relevant definitions. Let $G$ be an $n$-vertex $k$-graph. A \defn{perfect fractional matching} of $G$ is a function $\wx \colon E(G)\to [0,1]$ such that, for every vertex $v\in V(G)$, we have $\sum_{e\ni v}\wx(e)=1$.
Observe that this implies that $\sum_{e \in E(G)} \wx(e)=n/k$, which we will use throughout without mention.
For $\mu \in (0,1]$, a perfect fractional matching $\wx$ is said to be \defn{$\mu$-normal} if $\mu n^{-k+1} \le \wx(e) \le \mu^{-1}n^{-k+1}$ for all $e\in E(G)$.
The following result states that Dirac hypergraphs have normal perfect fractional matchings.

\begin{lemma}[Lemma 4.2 in~\cite{GGJKO:21}]\label{lem:normal PFM}
Let $1/n \ll \mu \ll \eps, 1/k$ with $k\in \bN$ and $k\ge 2$.
Let $G$ be an $n$-vertex $k$-graph with $\delta(G)\ge (1/2+\eps)n$. Then $G$ has a $\mu$-normal perfect fractional matching.
\end{lemma}

\subsection{Probabilistic tools}

We will often apply the following standard
Chernoff-type concentration inequality (see~\cite[Theorem~1.1]{DP:09}).

\begin{lemma}[Chernoff's inequality]
\label{lem:chernoff}
    Let $X_1,\dots,X_N$ be independent random variables taking values in $[0,1]$, and let $X = \sum_{i=1}^N X_i$. Then for every $0 < \beta < 1$,
    \[
    \PP\Big[|X - \EE[X]| \ge \beta \EE[X]\Big] \le 2\exp\left(-\frac{\beta^2}{3} \EE[X]\right)\, .
    \]
\end{lemma}

As explained at the end of Section~\ref{sec:sketch}, our proof will use a hypergraph matching argument and will need the matching to look random-like with respect to some properties.
In order to achieve that, we use a nibble-type result due to Ehard, Glock and Joos~\cite{EGJ:20}. For a hypergraph $\cH$, define $\Delta(\cH) := \max_{v \in V(\cH)} d_{\cH}(\{v\})$ and $\Delta^c(\cH) := \max_{\{u,v\} \subseteq V(\cH)} d_{\cH}(\{u,v\})$. We will consider edge weight functions $w\colon E(\mathcal{H}) \rightarrow \mathbb{R}_{\geq 0}$ and, for a set $A \subseteq E(\mathcal{H})$, we use the notation $w(A) := \sum_{e \in A}w(e)$.

\begin{theorem}[Theorem~1.2 in~\cite{EGJ:20}]\label{thm:EGJ}
Suppose $\delta\in(0,1)$ and $t\in \bN$ with $t\ge 2$, and set $\gamma:=\delta/50t^2$. Then there exists $\Delta_0$ such that for all $\Delta\ge \Delta_0$, the following holds.
Let $\cH$ be a $t$-uniform hypergraph satisfying $\Delta(\cH)\leq \Delta$, $\Delta^c(\cH)\le \Delta^{1-\delta}$ and $e(\cH)\leq \exp(\Delta^{\gamma^2})$. 
Let $\cW$ be a set of at most $\exp(\Delta^{\gamma^2})$ weight functions on~$E(\cH)$ such that $w(E(\cH))\ge \max_{e\in E(\cH)}w(e)\Delta^{1+\delta}$ for every $w \in \cW$.
Then there exists a matching $\cM$ in~$\cH$ such that $w(\cM)=(1\pm \Delta^{-\gamma}) w(E(\cH))/\Delta$ for every $w \in \cW$.
\end{theorem}

\section{Key lemma}\label{sec:key lemma}

In this section, we state and prove our key structural lemma. Recall from Section~\ref{sec:sketch} that a \emph{$k$-grid} is the $k$-graph on vertices $\{x_{ij}: 1 \le i,j \le k\}$ and with edges $x_{i1} \cdots x_{ik}$ for each $i \in [k]$ (which we will call \emph{horizontal} edges) and $x_{1j} \cdots x_{kj}$ for each $j \in [k]$ (which we will call \emph{vertical} edges).
An \emph{alternating $k$-grid} is a $2$-edge-coloured $k$-grid with all the horizontal edges red and all vertical edges blue.
A \emph{near-alternating $k$-grid} is a $2$-edge-coloured $k$-grid such that all horizontal edges but (at most) one are red, and all vertical edges but (at most) one are blue. Thus, the difference between a near-alternating $k$-grid and an alternating $k$-grid lies in not prescribing the colour of one horizontal and one vertical edge (cf.~Figure~\ref{fig_gadget}).

Our key lemma exploits the specific structure of the given colouring and shows that either it contains a near-alternating $k$-grid or the colouring is almost monochromatic.
While our main result applies to hypergraphs coloured with any number of colours, it suffices to handle here the case of only two colours.

\begin{lemma}[Key lemma]
\label{lem:key lemma} 
Let $1/n \ll \zeta \ll \eps, \rho, 1/k$ with $k\in \bN$ and $k\ge 3$.
Let $G$ be an $n$-vertex $k$-graph whose edges are $2$-coloured.
Assume that all but at most $\zeta n^{k-1}$ $(k-1)$-subsets $S$ of $V(G)$ satisfy $d_G(S) \ge (1/2+\eps)n$.
Then either there exists a near-alternating $k$-grid, or one of the colour classes has size at most $\rho n^k$.
\end{lemma}
\noindent
As explained in Section~\ref{sec:sketch}, Lemma~\ref{lem:key lemma} is enough to derive Corollary~\ref{cor:PM}.
Moreover, it is easy to see that the lemma also holds for $k=2$, but then a near-alternating grid is not a suitable gadget, because such a grid (which is a 4-cycle) might only have perfect matchings with one red edge and one blue edge, so that no colour appears more often.
Hence we ignore this case \nolinebreak here.

We now sketch the proof of Lemma~\ref{lem:key lemma}.
It is helpful to first guarantee that if a $(k-1)$-set of vertices is contained in an edge of a certain colour, then it is actually contained in many edges of this colour.
While this is not true for an arbitrary edge-colouring, we can make sure this is the case after deleting only few edges.
The required cleaning procedure is given by the following standard tool. We provide the short proof for the convenience of the reader.
\begin{prop}\label{prop:cleaning}
Let $G$ be a $k$-graph on $n$ vertices.
Then by removing at most $t\binom{n}{k-1}$ edges, one can ensure that in the resulting subhypergraph, every $(k-1)$-set has degree either $0$ or at least~$t$.
\end{prop}

\proof
As long as there is a $(k-1)$-set $S$ with degree between $1$ and $t$, delete all edges containing~$S$. Obviously, every $(k-1)$-set is considered at most once during this process, and when it is considered, we delete at most $t$ edges.
\endproof

Once $G$ has been cleaned (with respect to both colours and to a suitable choice of $t$), we obtain a subhypergraph $G'$ with the desired property for both colours.
Observe that, since we removed only few edges, for most of the $(k-1)$-subsets of $V(G)$, their degree in $G'$ will still be linearly above $n/2$.

Let $H$ be the $(k-1)$-shadow of $G'$, and equip $H$ with the following 2-colouring of its edges: Colour an edge of $H$ with colour $c$ if it is contained in at least one edge of $G'$ of colour $c$ (and hence at least $t$ such edges). We note that an edge of $H$ can receive both colours (if it is contained in an edge of $G'$ of each of the colours).
As we will show, by choosing $t$ appropriately, it suffices to find an alternating $(k-1)$-grid in $H$ in order to obtain a near-alternating $k$-grid in $G$ (see Claim~\ref{claim:almost_to_grid}).

If $H$ contains many double-coloured edges, then we can easily find an alternating $(k-1)$-grid.
In fact, we can use the following classical result of Erd\H{o}s~\cite{erdos:64b}, which states that the Tur\'an density of $k$-partite $k$-graphs is~$0$.

\begin{theorem}[\cite{erdos:64b}]\label{thm:zero_turan_density_partite_graph}
Let $1/n \ll \eta, 1/\ell$ with $\ell \in \mathbb{N}$.
Let $L$ be any $k$-partite $k$-graph with $\ell$ vertices.
Then any $k$-graph with $n$ vertices and at least $\eta n^k$ edges contains $L$ as a subgraph.
\end{theorem}

Call an edge of $H$ bad if it is double-coloured or has small degree in $G'$.
Then by Theorem~\ref{thm:zero_turan_density_partite_graph} and what we observed above, we can assume that only few edges of $H$ are bad. 
We would like to argue that all the edges of $H$ which are not bad are then coloured with the same unique colour, which in turn will imply that $G$ itself is almost-monochromatic.

We proceed as follows:
Let $\ordered{T_1}$ and $\ordered{T_2}$ be arbitrary orderings of any two edges $T_1, T_2 \in E(H)$ which are not bad.
The key observation we use is the following:
Suppose there is a tight walk in $G'$ connecting $\ordered{T_1}$ and $\ordered{T_2}$ such that any $(k-1)$ consecutive vertices of the walk give an edge of $H$ which is not bad. 
The power of the cleaning procedure then comes in handy, as we can claim that each $(k-1)$-set along the walk is contained in edges of $G'$ of the same unique colour.  
Therefore the colour information propagates from $T_1$ to $T_2$ and the tight walk must be monochromatic, implying that $T_1$ and $T_2$ must have the same colour.

To utilize the above observation, we want to be able to connect every (or almost every) pair of ordered edges $\ordered{T_1},\ordered{T_2}$, as above. 
The standard tool to obtain this connection is Lemma~\ref{lem:path connect}.
However, we cannot apply this lemma directly, as the required minimum degree does not hold in $G'$. 
Nevertheless, by randomly sampling a small set of vertices, we can avoid all the bad $(k-1)$-sets (as these are few) and thus apply Lemma~\ref{lem:path connect} within the sample. To make this work, we perform another cleaning step which removes $(k-1)$-sets that intersect with too many \nolinebreak bad \nolinebreak sets.

We are now ready to prove the key lemma.

\lateproof{Lemma~\ref{lem:key lemma}}
Let $k \in \bN$ with $k \ge 3$ and $\eps,\rho>0$ be given, and observe that we can assume $\eps$ to be small enough for Lemma~\ref{lem:path connect} to hold.
Let
$$1/n \ll \zeta \ll \eta \ll \eps, \rho, 1/k\, .$$

Let $G$ be a $2$-edge-coloured (with red and blue) $n$-vertex $k$-graph on $V$ such that all but at most $\zeta n^{k-1}$ of the $(k-1)$-subsets of $V$ satisfy $d_G(S) \ge (1/2+\eps)n$.
We apply Proposition~\ref{prop:cleaning} with parameter $t:=\eta \eps n$ twice, once to the subhypergraph induced by the red edges and once to the subhypergraph induced by the blue edges.
This gives a subhypergraph $G'$ of $G$ with $e(G') \geq e(G) - 2\eta \eps n^k$ such that each $(k-1)$-set $S$ satisfies the following condition for each of the colours: either $S$ is not contained in any edge of $G'$ of this colour, or it is contained in at least $\eta \eps n$ edges of $G'$ of this colour. 

Let $H$ be the $(k-1)$-uniform shadow of $G'$, equipped with the following $2$-colouring of its edges: 
Colour an edge of $H$ red (resp. blue) if it is contained in at least one red (resp. blue) edge of $G'$ (and hence at least $\eta \eps n$ such edges), noting that we allow edges of $H$ to receive both colours.
For the purpose of finding coloured structures in $H$, an edge that has both colours can be used either way.

\begin{claim}\label{claim:almost_to_grid}
    If $H$ contains an alternating $(k-1)$-grid, then $G$ contains a near-alternating $k$-grid.
\end{claim}

\begin{proof}[Proof of Claim~\ref{claim:almost_to_grid}]
    Suppose $H$ contains an alternating $(k-1)$-grid.
    Then there exists $W:=\{x_{ij}:1 \le i,j \le k-1\} \subseteq V$ such that $x_{i1} \cdots x_{i(k-1)}$ is a red edge of $H$ for each $i \in [k-1]$ and $x_{1j} \cdots x_{(k-1)j}$ is a blue edge of $H$ for each $j \in [k-1]$.
    Let $\cR$ be the set of ordered tuples $(x_{1k},x_{2k},\dots,x_{(k-1)k})$ such that $x_{1k},x_{2k},\dots,x_{(k-1)k} \in V \setminus W$ are pairwise distinct, $\{x_{i1} \cdots x_{ik}\}$ is a red edge of $G$ for each $i \in [k-1]$, and $d_G(\{x_{1k},x_{2k},\dots,x_{(k-1)k}\}) \ge (1/2+\eps)n$.
    Owing to the fact that a red edge of $H$ can be extended to a red edge of $G'$ (and hence $G$) in at least $\eta \eps n$ ways and that for all but $\zeta n^{k-1}$ of the $(k-1)$-subsets $S$ of $V$ it holds that $d_G(S) \ge (1/2+\eps)n$, we have 
    $|\cR| \ge (\eta \eps n)^{k-1} - n^{k-2} - (k-1)! \cdot \zeta n^{k-1} \ge (\eta \eps n)^{k-1}/2$, where we used $\zeta \ll \eta, \eps,1/k$ (and the $n^{k-2}$ term accounts for the choices of $x_{1k},\dots,x_{(k-1)k}$ which are not pairwise distinct).
    Similarly, let $\cB$ be the set of ordered tuples $(x_{k1},x_{k2},\dots,x_{k(k-1)})$ such that $x_{k1},x_{k2},\dots,x_{k(k-1)} \in V \setminus W$ are pairwise distinct, $\{x_{1j} \cdots x_{kj}\}$ is a blue edge of $G$ for each $j \in [k-1]$, and $d_G(\{x_{k1},x_{k2},\dots,x_{k(k-1)}\}) \ge (1/2+\eps)n$.
    As above for $\cR$, we have $|\cB| \ge (\eta \eps n)^{k-1}/2$.
    Therefore, there exist vertex-disjoint 
    $R \in \cR$ and $B \in \cB$. 
    As $d_G(R),d_G(B) \ge (1/2+\eps)n$, 
    we have $|N_G(R) \cap N_G(B)|\ge 2(1/2+\eps)n - n = 2\eps n$. Hence, there exists $x_{kk} \in (N_G(R) \cap N_G(B)) \setminus W$, i.e.~a vertex $x_{kk}$ which forms an edge of $G$ with both $R$ and $B$ (although we have no control of the colours of these edges).
    This gives a near-alternating $k$-grid in $G$, as desired.    
\end{proof}

Owing to Claim~\ref{claim:almost_to_grid}, from now on we can assume that $H$ does not contain an alternating $(k-1)$-grid.
We show that then $H$ (and thus $G$) must be almost monochromatic.
A $(k-1)$-subset $S$ of $V$ is called \emph{bad} if $d_{G'}(S) < (1/2+\eps/2)n$ or if, seen as an edge of $H$, $S$ is coloured with both colours.
We bound the number of bad sets as follows.

\begin{claim}\label{claim:bad sets}
There are at most $5k\eta n^{k-1}$ bad sets. 
\end{claim}
\begin{proof}[Proof of Claim~\ref{claim:bad sets}]
We begin by bounding the number of $(k-1)$-sets $S$ with $d_{G'}(S) < (1/2+\eps/2)n$.
Observe that if $d_{G'}(S) < (1/2+\eps/2)n$, then either $d_G(S) < (1/2+\eps)n$ (i.e.~$S$ has small degree already in $G$), or we removed at least $\eps n/2$ edges of $G$ containing $S$ during the cleaning process (i.e.~when obtaining $G'$ from $G$). 
The former case holds for at most $\zeta n^{k-1}$ sets $S$, by assumption. Let us now bound the number of $S$ in the latter case. 
Since removing an edge of $G$ decreases (by one) the degree of precisely $k$ $(k-1)$-sets, and as the total number of removed edges is at most $2\eta \eps n^k$, the number of such sets $S$ is at most $\frac{k \cdot 2 \eta \eps n^k}{\eps n/2} = 4 k \eta n^{k-1}$.

Next, we bound the number of double-coloured edges of $H$.
    Let $\tilde{H}$ be the subhypergraph of $H$ on $V$ induced by the double-coloured edges.
    Then $e(\tilde{H}) \le \eta n^{k-1}$.
    Indeed, otherwise $\tilde{H}$ contains a copy of the $(k-1)$-grid by Theorem~\ref{thm:zero_turan_density_partite_graph}, 
    since the $(k-1)$-grid is a $(k-1)$-partite $(k-1)$-graph (with each vertex class of size $k-1$).
    However, as each edge of $\tilde{H}$ receives both colours, this gives an alternating $(k-1)$-grid, which is a contradiction.

Summarizing, the number of bad sets is at most 
$(\zeta + 4 k \eta +\eta)n^{k-1} \le 5k\eta n^{k-1}$, 
where we used $\zeta \ll \eta$.  
\end{proof}

Given a $(k-1)$-subset $S$ of $V$, we say that $S$ is \emph{clean} if, for every $0\le j\le k-1$, the number of bad sets $T$ with $|S\cap T|=j$ is at most $\eta^{1/2} n^{k-1-j}$.
We remark that for $j=k-1$ the condition means that $S$ itself is not bad.
\begin{claim}\label{claim:clean_sets}
    All but at most $\eta^{1/3} n^{k-1}$ of the $(k-1)$-subsets of $V$ are clean.
\end{claim}
\begin{proof}[Proof of Claim~\ref{claim:clean_sets}]
    For each $0 \le j \le k-1$, we bound the number of $(k-1)$-sets $S$ for which the number of bad sets $T$ with $|S \cap T| = j$ is more than $\eta^{1/2} n^{k-1-j}$.
    If $j=0$, this cannot happen by Claim~\ref{claim:bad sets}, so we can assume $0 < j \le k-1$.
    Given a bad set $T$, the number of $(k-1)$-sets $S$ intersecting $T$ in $j$ vertices is at most $\binom{k-1}{j}n^{k-1-j}$.
    Therefore, using the bound in Claim~\ref{claim:bad sets}, the number of $(k-1)$-subsets of $V$ which are not clean is at most
    \begin{equation*}
        \sum_{j=1}^{k-1} \frac{\binom{k-1}{j} n^{k-1-j} \cdot 5 k \eta n^{k-1}}{\eta^{1/2} n^{k-1-j}} \le \eta^{1/3} n^{k-1}\, . 
    \end{equation*}
\end{proof}
Let $H'$ be the $(k-1)$-subhypergraph of $H$ consisting of the edges which are clean $(k-1)$-sets.
Observe in particular that each edge of $H'$ has a unique colour (because double-coloured edges are bad).
 
\begin{claim}
\label{claim:unique_colour_H'}
    All edges of $H'$ have the same colour.
\end{claim}

Using Claim~\ref{claim:unique_colour_H'}, we can easily complete the proof of the lemma.
Indeed, suppose without loss of generality that the edges of $H'$ are red.
Then for an edge of $G$ to be blue, it has to be either an edge of $G \setminus G'$, or an edge of $G'$ none of whose $(k-1)$-subsets belongs to $H'$.
Using that $e(G) - e(G') \leq 2 \eta \eps n^k$ and $e(H) - e(H') \leq \eta^{1/3} n^{k-1}$ (by Claim~\ref{claim:clean_sets}), we get
that the number of blue edges of $G$ is at most $2 \eta \eps n^k + n \cdot \eta^{1/3} n^{k-1} \le \rho n^k$, as wanted.
We are left to prove Claim~\ref{claim:unique_colour_H'}.

\begin{proof}[Proof of Claim~\ref{claim:unique_colour_H'}]
    Consider two arbitrary edges $T_1, T_2 \in E(H')$. We will show that they have the same colour.
    Fix $C,\beta$ with $\eta \ll 1/C \ll \beta \ll \eps, \rho, 1/k$. 
    We claim that there exists a set $R\In V$ such that
\begin{enumerate}[label=(\roman*)]
    \item \label{key_lemma_(i)} $|R| \ge C/2$;
    \item \label{key_lemma_(ii)} for $i=1,2$, for every $(k-1)$-set $S$ which is contained in $R \cup T_i$ and is not bad, it holds that $|N_{G'}(S) \cap R|\ge (1/2+\eps/8)|R \cup T_i|$;
    \item \label{key_lemma_(iii)} no $(k-1)$-subset contained in $R \cup T_1$ or $R \cup T_2$ is bad.
\end{enumerate}

Let $R \subseteq V$ be obtained by independently including each vertex of $V$ with probability $C/n$.
We show that such $R$ satisfies~\ref{key_lemma_(i)},~\ref{key_lemma_(ii)} and~\ref{key_lemma_(iii)} with positive probability.

Since $\EE[|R|]=C$ and $1/C \ll \beta$, an easy application of Chernoff's inequality (Lemma~\ref{lem:chernoff}) shows that $(1-\beta) C \le |R| \le (1+\beta) C$ with probability at least $0.9$.

Fix $i \in \{1,2\}$, let $S \subseteq V$ be a $(k-1)$-set which is not bad, and define the events $A_S:= \{S \subseteq R \cup T_i\}$ and $B_S:=\{X_S < (1/2+\eps/4)C\}$, where $X_S := |N_{G'}(S) \cap R|$.
Furthermore, let $j=j_S:=|S \cap T_i|$, so $0 \le j \le k-1$.
We are going to show that with probability at least $0.9$, the event $A_S \cap B_S$ does not hold for any such $S$.
Note that $A_S$ and $B_S$ are independent, and $\PP[A_S]=\left(\frac{C}{n}\right)^{k-1-j}$.

Since $S$ is not bad, we have $\EE[X_S] = |N_{G'}(S)| \cdot \frac{C}{n} \ge (1/2 + \eps/2) C$ and, using Chernoff's inequality (Lemma~\ref{lem:chernoff}), we conclude that $\PP[B_S] \le 2 \exp\left(-\frac{(\varepsilon/4)^2}{3} \EE[X_S]\right)=\exp(-\Omega_{\eps}(C))$. 
Therefore, $\PP[A_S \cap B_S] \le \left(\frac{C}{n}\right)^{k-1-j} \cdot \exp(-\Omega_{\eps}(C))$.
By taking the union bound over the at most $2^{k-1}n^{k-1-j}$ $(k-1)$-sets $S \subseteq V$ with $j_S=j$, we see that the probability that $A_S \cap B_S$ holds for some such $S$ is at most $\left(\frac{C}{n}\right)^{k-1-j} \cdot \exp(-\Omega_{\eps}(C)) \cdot 2^{k-1}n^{k-1-j} \le 0.1 k^{-1}$, where we used that $1/C \ll \eps, 1/k$.
The conclusion now follows by taking another union bound over the $k$ choices for~$j$.

Next, let $Y$ be the random variable counting the number of bad sets contained in $R \cup T_i$, and observe that $Y= \sum_{j=0}^{k-1} Y_j$, where $Y_j$ counts the number of bad sets $T \subseteq R \cup T_i$ with $|T \cap T_i|=j$.
Since $T_i$ is a clean set, the number of bad sets $T$ with $|T \cap T_i|=j$ is at most $\eta^{1/2} n^{k-1-j}$.
Therefore, $\EE[Y_j] \le \eta^{1/2} n^{k-1-j} \cdot \left( \frac{C}{n} \right)^{k-1-j} \leq \eta^{1/2}C^{k-1}$ and hence $\EE[Y] \le k \eta^{1/2} C^{k-1}\le 0.1$, using $\eta \ll 1/C$. By Markov's inequality, we have that $Y=0$ with probability at least $0.9$.

Therefore, with positive probability we have that $(1-\beta)C \le |R| \le (1+\beta)C$, the event $A_S \cap B_S$ does not hold for any $(k-1)$-set which is not bad (for both $i=1,2$), and no $(k-1)$-subset contained in $R \cup T_1$ or $R \cup T_2$ is bad. These in turn imply properties~\ref{key_lemma_(i)}-\ref{key_lemma_(iii)}.
Indeed,~\ref{key_lemma_(i)} and~\ref{key_lemma_(iii)} follow directly from the above, and for~\ref{key_lemma_(ii)} it is enough to observe that if $S$ is not bad and is contained in $R \cup T_i$ (i.e.~$A_S$ holds), then $B_S$ cannot hold and thus $X_S \ge (1/2+\eps/4)C \ge (1/2+\eps/8)|R \cup T_i|$, where the first inequality follows from the definition of $B_S$, and the second inequality uses $|R| \le (1+\beta)C$, $|T_i| = k-1$, and $1/C, \beta \ll \eps,1/k$.

We conclude that a set $R$ satisfying~\ref{key_lemma_(i)},~\ref{key_lemma_(ii)} and~\ref{key_lemma_(iii)} does indeed exist. Note that~\ref{key_lemma_(ii)}--\ref{key_lemma_(iii)} imply that $\delta(G'[R \cup T_i]) \geq (1/2 + \varepsilon/8)|R \cup T_i|$. 
Moreover, by~\ref{key_lemma_(i)}, $|R|$ is large enough to apply Lemma~\ref{lem:path connect} to $G'[R \cup T_1]$ and $G'[R \cup T_2]$. 
Now, fix two arbitrary orderings $\ordered{T_1}$ and $\ordered{T_2}$ of $T_1, T_2$, respectively, and let $\ordered{T}$ be an arbitrary ordering of a $(k-1)$-set $T \subseteq R \setminus (T_1 \cup T_2)$.
Using Lemma~\ref{lem:path connect} twice, we find a tight path  $P_1$ connecting $\ordered{T_1}$ and $\ordered{T}$ in $G'[R \cup T_1]$ and a tight path $P_2$ connecting $\ordered{T_2}$ and $\ordered{T}$ in $G'[R \cup T_2]$.
Owing to~\ref{key_lemma_(iii)}, 
we know that every $k-1$ consecutive vertices in $P_1$ form a set which is not bad, which means that all edges of $G'$ containing this set have the same colour. By the definition of the colouring of $H$, every such $(k-1)$-set has the same colour and, in particular, this holds for $T_1$ and $T$. 
By repeating the same argument for $P_2$, this holds for $T_2$ and $T$.
We conclude that $T_1$ and $T_2$ have the same colour, as desired.
\end{proof}
\noindent
This concludes the proof of Lemma~\ref{lem:key lemma}.
\endproof

\section{Perfect fractional matchings with high discrepancy}\label{sec:fractional}
In this section we focus on the random walk which we will use to sample a collection of paths of high discrepancy. 
But first, recalling the definition of perfect fractional matchings from Section~\ref{sec:dirac_graphs}, we introduce the following notation: For a perfect fractional matching $\wx \colon E(G)\to [0,1]$ of a $k$-graph $G$, and for a set $S\subseteq V(G)$ with $|S|\le k$, define 
$\wx(S):=\sum_{e \in E(G)\colon S \subseteq e}\wx(e)$.
Note that $\wx(\Set{v})=1$ for all $v\in V(G)$ and $\wx(S)=0$ for all $S \subseteq V(G)$ with $|S|=k$ and $S \not\in E(G)$. 

We define the following random walk $\cY=(Y_1, Y_2, \dots)$ on $V:=V(G)$.
It begins with an ordered $(k-1)$-tuple $(Y_1,\dots,Y_{k-1}) \in (V)_{k-1}$ chosen according to the following \emph{initial distribution} $\pi\colon (V)_{k-1} \rightarrow [0,1]$.
Pick an ordered $(k-1)$-set $(Y_1,\dots,Y_{k-1}) \in (V)_{k-1}$ at random with probability proportional to $\wx(\cdot)$, that is, for any $\ordered{S} \in (V)_{k-1}$, 
\begin{equation}
\label{eq:RW_initial_distribution}
    \pi(\ordered{S}):= \frac{\wx(S)}{\sum_{\ordered{S'} \in (V)_{k-1}} \wx(S')}\, .
\end{equation}
Observe that the denominator of~\eqref{eq:RW_initial_distribution} can be rewritten as 
\begin{equation}
\label{eq:RW_denominator}
    \sum_{\ordered{S'} \in (V)_{k-1}} \wx(S') = k! \cdot \sum_{e \in E(G)} \wx(e) = (k-1)! \cdot n\, ,
\end{equation}
where we used that every edge contains $k!$ ordered $(k-1)$-sets.
The \emph{transition probability} will also be defined according to $\wx$.
For all $i \ge k-1$, conditional on the outcome of $Y_{i-(k-2)},\dots,Y_i$, we choose the next vertex $Y_{i+1}$ as follows: Let $\ordered{Z_i}:=(Y_{i-(k-2)},\dots,Y_i)$ be the ordered set of the last $k-1$ vertices in the sequence and choose $Y_{i+1}$ with probability proportional to $\wx(Z_i \cup \{\cdot\})$, that is, for any $v \in V \setminus Z_i$,
\begin{equation}
\label{eq:RW_transition}
    \Pr[Y_{i+1}=v|Y_{i-(k-2)},\dots,Y_{i}] = \frac{\wx(Z_{i}\cup \Set{v})}{\sum_{v' \in V \setminus Z_i}\wx(Z_i \cup \Set{v'})} = \frac{\wx(Z_{i}\cup \Set{v})}{\wx(Z_i)} \, ,
\end{equation}
and for any $v \in Z_i$ the transition probability is $0$.
Observe that $\cY$ is equivalent to the random walk $\cZ:=(\ordered{Z_{k-1}},\ordered{Z_k},\dots)$, and we can refer to both. 
In fact, with $(V)_{k-1}$ viewed as the state space, $\cZ$ is a Markov chain
and it is then easy to check that the distribution $\pi$ defined in~\eqref{eq:RW_initial_distribution} is stationary (cf.~\cite[Proposition~5.5]{GGJKO:21}). The important fact of defining the transition probabilities in terms of the perfect fractional matching $\wx$ is that then the random walk behaves uniformly with respect to the visited vertices, in the sense that the distribution of the vertices for $\cY$ to visit at any step is uniform over $V(G)$ as proved below.

\begin{fact}
\label{fact:RW_uniform_vertices}
    For each integer $i \ge 1$ and $v \in V(G)$, we have that $\Pr[Y_i=v]=1/n$.
    Moreover, for each $k \le i \le t$, with $e_i:=\{Y_{i-k+1},\dots,Y_i\}$ denoting the $(i-k+1)$-st edge of $\cY$, the following holds:
    For each $e\in E(G)$, we have that $\Pr[e_i=e]=\frac{k}{n}\wx(e)$.
\end{fact}

\proof
Let $v\in V$. For $1 \le i \le k-1$, observe the following identity where the first sum runs over all $\ordered{S} \in (V)_{k-1}$ such that $v$ is the $i$-th element of $\ordered{S}$, and the second sum runs over all $S \in \binom{V}{k-1}$ with $v \in S$:
    \[
        \sum_{\substack{\ordered{S} = (v_1,\dots,v_{k-1}) \\ v_i = v}} \wx(S) = (k-2)! \cdot 
        \sum_{S \in \binom{V}{k-1} : \; v \in S} \wx(S) = (k-1)! \cdot \sum_{e\,:\,e \ni v} \wx(e) = (k-1)! \,.
    \]
    Together with~\eqref{eq:RW_denominator}, we get $\Pr[Y_i=v] = 1/n$.

    Since $\pi$ as defined in~\eqref{eq:RW_initial_distribution} is the stationary distribution of $\cZ$, it holds that $\Pr[\ordered{Z_i}=\ordered{S}] = \pi(\ordered{S})$ for each $i \ge k-1$.
    Let $i \ge k$.
    Then by the law of total probability we have
    \begin{align*}
        \Pr[Y_i=v] &=\sum_{\ordered{S} \in (V)_{k-1}} \Pr\left[Y_i=v | \ordered{Z_{i-1}}=\ordered{S}\right] \cdot \Pr\left[\ordered{Z_{i-1}}=\ordered{S}\right] \\
        &= \sum_{\ordered{S} \in (V)_{k-1}} \Pr\left[Y_i=v | \ordered{Z_{i-1}}=\ordered{S}\right] \cdot \pi(\ordered{S}) \\
        &= \sum_{\substack{\ordered{S} \in (V)_{k-1}: \\ v \not\in S}} \frac{\wx(S \cup \{v\})}{\cancel{\wx(S)}} \cdot \frac{\cancel{\wx(S)}}{(k-1)! \cdot n} = 
        \frac{(k-1)! \cdot \sum_{e \ni v} \wx(e)}{(k-1)! \cdot n} = \frac{1}{n}\, ,
    \end{align*}
    where we used~\eqref{eq:RW_denominator}, that for an edge $e$ with $e \ni v$ there are $(k-1)!$ choices for $\ordered{S} \in (V)_{k-1}$ such that $S \cup \{v\}=e$, and that $\wx$ is a perfect fractional matching.

    For the ``morever"-part of Fact~\ref{fact:RW_uniform_vertices}, let $k \le i \le t$ and $e \in E(G)$.
    Then 
    \begin{align*}
        \Pr\left[e_i = e \right] &= \sum_{\ordered{S} \in (V)_{k-1}} \Pr\big[e_i=e| \ordered{Z_{i-1}}=\ordered{S}\big] \cdot \Pr\big[\ordered{Z_{i-1}}=\ordered{S}\big] \\
        &= \sum_{\ordered{S} \in (V)_{k-1}\colon S\In e} \frac{\wx(e)}{\cancel{\wx(S
        )}} \cdot \frac{\cancel{\wx(S)}}{(k-1)! \cdot n} =\frac{k}{n} \cdot \wx(e)\, ,
    \end{align*} 
    where we used~\eqref{eq:RW_denominator} and that there are $k!$ ways to choose $\ordered{S} \in (V)_{k-1}$ such that $S\In e$. 
\endproof

Despite Fact~\ref{fact:RW_uniform_vertices}, for an arbitrary perfect fractional matching $\wx$, the behaviour of the random walk $\cY$ can still be quite trivial.
For instance, suppose that $\wx$ is indeed a perfect matching $M$, that is $\wx(e)=\IND(e\in M)$.
Then, once the first ordered $(k-1)$-set $\ordered{Z_{k-1}}$ is chosen, the walk is completely deterministic (and uses the same edge in each step).
In order to avoid this in a robust way, we will assume that $\wx$ is $\mu$-normal for some (small) constant $\mu > 0$ (see Section~\ref{sec:dirac_graphs} for the definition).

Ultimately, we would like to use the random walk $\cY$ to sample a collection of tight paths which cover almost all the vertices of $G$ and have high discrepancy.
As proved in Fact~\ref{fact:RW_uniform_vertices}, the probability that $\cY$ sees a certain edge $e$ is proportional to $\wx(e)$.
Therefore, in order to guarantee that $\cY$ sees a substantial number of edges of the same colour, it will be enough that $\wx$ is a normal perfect fractional matching with high discrepancy.
Recall that the total weight assigned by any perfect fractional matching is $n/k$, so, just by averaging, some colour will receive a total weight of at least $n/(rk)$.
By combining Lemma~\ref{lem:normal PFM} with our key lemma (Lemma~\ref{lem:key lemma}), we are able to boost the discrepancy.
Starting with a perfect fractional matching given by Lemma~\ref{lem:normal PFM}, if we can find a near-alternating $k$-grid (which we refer to as a {\em gadget} from now on), then by increasing the weight of the red matching, say, and decreasing the weight of the blue matching by the same amount, the total weight of each vertex remains unchanged, but the total weight of the red edges has increased. Obviously, one gadget will only allow us to perform an insignificant perturbation, but by applying the key lemma iteratively, we can find many edge-disjoint gadgets and together they allow us to perturb the initial perfect fractional matching by a significant amount.

\begin{lemma}
\label{lem:normal PFM_discrepancy}
Let $1/n \ll \mu \ll \eps, 1/r, 1/k$ with $k,r\in \bN$, $k\ge 3$ and $r\ge 2$.
Let $G$ be an $n$-vertex $k$-graph with $\delta(G)\ge (1/2+\eps)n$ whose edges are $r$-coloured. Then $G$ has a $\mu$-normal perfect fractional matching such that the total weight received by some colour class is at least $(1+ \mu)\frac{n}{rk}$.
\end{lemma}

\proof
Let 
$$1/n \ll \mu \ll \eta \ll \zeta \ll \rho \ll \mu_0 \ll \eps, 1/r, 1/k\, ,$$
where $\mu_0$ is small enough for Lemma~\ref{lem:normal PFM} to hold on input $k$ and $\eps$, and $\zeta$ is small enough for Lemma~\ref{lem:key lemma} to hold on input $k,\eps/2$ and $\rho/2$. Let $\wx_0$ be a $\mu_0$-normal perfect fractional matching as given by Lemma~\ref{lem:normal PFM}.

Let $c_1, \dots, c_r$ be the $r$ colours used in the edge-colouring of $G$.
Since Lemma~\ref{lem:key lemma} is only stated for $2$-edge-coloured graphs, we will now consider the following $2$-edge-colouring of $G$, obtained by identifying $r-1$ of the colours to a unique colour:
Colour by ``red'' any edge coloured by $c_1$, and by ``blue'' every other edge.

We claim that either $G$ contains $\eta n^k$ edge-disjoint gadgets or one of the colour classes of $G$ has size at most $\rho n^k$.
This follows by applying Lemma~\ref{lem:key lemma} iteratively. 
Suppose indeed that a maximal collection of edge-disjoint gadgets has size $\ell < \eta n^k$, and let $G'$ be the subhypergraph of $G$ obtained by removing all the edges of such gadgets.
Let $\cS$ be the collection of the $(k-1)$-subsets $S$ of $V$ with $d_{G'}(S) < (1/2+\eps/2)n$.
We now bound the size of $\cS$. The total number of removed edges is $2k \ell<2k\eta n^k$. In order to have $S \in \cS$, we must have removed at least $\eps n/2$ edges containing~$S$, and each removed edge decreases the degree of $k$ $(k-1)$-sets by one.
Therefore $|\cS| \le \frac{2k^2 \eta n^k}{\eps n/2} \le \zeta n^{k-1}$, where we used $\eta \ll \zeta, \eps$ for the last inequality.
Since the collection was maximal, invoking Lemma~\ref{lem:key lemma}, we get that one of the colour classes of $G'$ has size at most $\rho n^k/2$.
Therefore one of the colour classes of $G$ has size at most $\rho n^k/2 + 2k \cdot \ell < \rho n^k$, where we used $\ell < \eta n^k$ and $\eta \ll \rho$.

If one of the colour classes of $G$, say $\cC$, has size at most $\rho n^k$, then, since $\wx_0$ is $\mu_0$-normal, this colour class gets a total weight of at most $\sum_{e \in \cC} \wx_0(e) \le |\cC| \mu_0^{-1} n^{-k+1} \le \mu_0^{-1} \rho n$.
Then the other colour class gets a total weight of at least $(1/k- \mu_0^{-1} \rho)n$.
By averaging, $\wx_0$ assigns to one of the colour classes of the original $r$-edge-colouring of $G$ a total weight of at least $\frac{(1/k-\mu_0^{-1}\rho)n}{r-1}  \ge (1+\mu) \frac{n}{rk}$, where we used $\rho \ll \mu_0, 1/r, 1/k$.
In particular, $\wx_0$ is already a desired $\mu$-normal perfect fractional matching with high discrepancy.

We are left with the case where there exists a collection $\cL:=\{L_i:i \in [\eta n^k]\}$ of $\eta n^k$ edge-disjoint gadgets.
Now, either the total weight assigned by $\wx_0$ to the red edges is at least $\frac{n}{rk}$, or the total weight assigned by $\wx_0$ to the blue edges is at least $\frac{(r-1)n}{rk}$.

Suppose we are in the first case.
We will modify $\wx_0$ on the edges of each gadget in $\cL$ to obtain a $\mu$-normal perfect fractional matching $\wx$ with high discrepancy in colour $c_1$.
For $L \in \cL$, let $e_1^L,\dots,e_k^L$ (resp. $f_1^L,\dots,f_k^L)$ be the horizontal (resp. vertical) edges of $L$, and recall these are pairwise distinct.
By the definition of a near-alternating $k$-grid, all edges $e_1^L,\dots,e_k^L$ but at most one are red, and all edges $f_1^L,\dots,f_k^L$ but at most one are blue. 
Define $\wx \colon E(G) \to [0,1]$ as follows:
$\wx(e_i^L)=\wx_0(e_i^L)+ \mu_0 n^{-k+1}/2$ and $\wx(f_i^L)=\wx_0(f_i^L)- \mu_0 n^{-k+1}/2$ for each $i \in [k]$ and $L \in \cL$.
Moreover, set $\wx(e)=\wx_0(e)$ for any other edge $e \in E(G)$.
In other words, $\wx$ is obtained from $\wx_0$ by decreasing the weight of each vertical edge by $\mu_0 n^{-k+1}/2$ and increasing the weight of each horizontal edge by the same quantity, in each of the gadgets in $\cL$. Observe that if a vertex is contained in a gadget, then it belongs to precisely one horizontal edge and one vertical edge of this gadget.
Therefore for every vertex $v \in V(G)$ we have that $\sum_{e \ni v} \wx(e) = \sum_{e \ni v} \wx_0(e) = 1$. Also, for every $e \in E(G)$ we have $\mu n^{-k+1} \le \wx_0(e)-\mu_0 n^{-k+1}/2 \le \wx(e) \le \wx_0(e)+\mu_0 n^{-k+1}/2 \le \mu^{-1} n^{-k+1}$, where we used that $\mu_0 n^{-k+1} \le \wx_0(e) \le \mu_0^{-1} n^{-k+1}$ and $\mu \ll \mu_0$.
This shows that $\wx$ is a $\mu$-normal perfect fractional matching.
Moreover, for each $L \in \mathcal{L}$, the weight given by 
$\wx$ to the red edges in $L$ is bigger by at least $\mu_0 n^{-k+1}/2$ than the weight given by $\wx_0$ to these edges. This is because we increased (by $\mu_0 n^{-k+1}/2$) the weight of at least $k-1 \geq 2$ red edges, and decreased (by the same amount) the weight of at most one such edge. As the gadgets in $\mathcal{L}$ are edge-disjoint and $|\mathcal{L}| = \eta n^k$, 
we conclude that the total weight assigned by $\wx$ to the red edges is at least $\frac{n}{rk}+\eta n^k \mu_0 n^{-k+1} /2\ge (1+\mu)\frac{n}{rk}$, using that $\mu \ll \eta,\mu_0$.
Therefore, the total weight assigned by $\wx$ to the colour class of $c_1$ in the original edge-colouring of $G$ is at least $(1+\mu)\frac{n}{rk}$.

Suppose now that we are in the second case, namely, that the 
total weight assigned by $\wx_0$ to the blue edges is at least $\frac{(r-1)n}{rk}$. We then use the same argument as above to find a $\mu$-normal perfect fractional matching $\wx$ which assigns a total weight of at least $(r-1)(1+\mu)\frac{n}{rk}$ to the blue edges. 
Recall that the blue edges are precisely those coloured by $c_2,\dots,c_r$ in the original colouring of $G$, and thus, by averaging, there is $2 \le i \le r$ such that the total weight assigned by $\wx$ to the colour class of $c_i$ in the original colouring of $G$ is at least $(1+\mu)\frac{n}{rk}$. 
\endproof

Ideally, we would like to use the random walk $\mathcal{Y}$ and Lemma~\ref{lem:normal PFM_discrepancy} to obtain long tight paths with high discrepancy.
However, one remaining problem is that, if we let the random walk continue for too many steps, then with high probability, it will not be self-avoiding anymore. 
It might be possible to analyse the ``self-avoiding'' version of this random walk and show that with high probability, it will cover almost all the vertices and will have high discrepancy.
However, such an analysis might be very intricate. To circumvent this, we prove the following ``sampling lemma'' which allows us, using the random walk, to sample from the set of all tight paths in $G$ of order $t$, where $t$ is a (large) constant, in such a way that every vertex appears in the chosen path with approximately the same probability, and we expect more edges of one specific colour.
In Section~\ref{sec:linear_forest}, we will then produce a large collection of paths sampled from this distribution, and then use a nibble-type argument to select from this collection a large linear forest with high discrepancy.

\begin{lemma}\label{lem:sampling}
    Let $1/n \ll 1/t, \mu \ll \eps, 1/k, 1/r$ with $k,r,t \in \bN$, $k\ge 3$ and $r\ge 2$.
    Let $G$ be an $n$-vertex $k$-graph with $\delta(G)\ge (1/2+\eps)n$ whose edges are $r$-coloured. Let $\Omega$ be the set of all tight paths of order $t$ in $G$. Then, there exists a colour ``red'' and a probability distribution on $\Omega$ such that a randomly chosen element $P\in \Omega$ has the following properties:
    \begin{enumerate}[label=\rm{(\arabic*)}]
        \item \label{eq:sampling_1} for any given tight path $Q$, we have $\prob{P=Q} \le O_{t,\mu}(n^{-t})$;
        \item \label{eq:sampling_2} for every $v\in V(G)$, we have $\Big|\prob{v\in V(P)} - \frac{t}{n}\Big| \le O_{t,\mu}\left(n^{-2}\right)$;
        \item \label{eq:sampling_3} the expected number of red edges in $P$ is at least $\frac{1+\mu}{r}\cdot (t-k+1)-O_{t,\mu}(n^{-1})$.
\end{enumerate}
\end{lemma}

\proof
    Let $\wx$ be a $\mu$-normal perfect fractional matching of $G$ such that the total weight received by some colour class, say ``red'', is at least $(1+\mu)\frac{n}{rk}$. This exists by Lemma~\ref{lem:normal PFM_discrepancy}.
    Let $\cY=(Y_1,Y_2,\dots)$ be the random walk defined via $\wx$, with notation as introduced at the beginning of this section. (In particular, we use $\Pr$ as the probability measure corresponding to the random walk, whereas $\PP$ will denote the desired probability measure on~$\Omega$.)
    Here, we will only be interested in the first $t$ vertices of $\cY$ and we note that those form a tight walk of order~$t$.
    Now, sample elements of $\Omega$ according to the first $t$ vertices of $\cY$, conditioned on $\cY$ being self-avoiding up to $Y_t$.
    More precisely, let $\cB$ be the event that $Y_1, \dots, Y_t$ are pairwise distinct, and denote by $Q$ any tight path of order $t$ and by $q_1, \dots, q_t$ the vertices of $Q$ (appearing in this order). Then take the distribution on $\Omega$ where a randomly chosen element $P \in \Omega$ satisfies
    \[
        \PP[P=Q] = \Pr\Big[\{Y_1=q_1,\dots,Y_t=q_t\} \cup \{Y_1=q_t,\dots,Y_t=q_1\}\; |\; \cB\Big]\, ,
    \]
    where we considered both orders of $Q$ since the elements of $\Omega$ are unordered.
    We claim that this distribution on $\Omega$ has the desired properties.
    Before proving that this is the case, we need some preliminary observations.
    
    Since $\wx$ is $\mu$-normal, the probability that the walk starts with $\ordered{S}=(q_1,\dots,q_{k-1})$ is 
    \[
        \pi(\ordered{S}) = \frac{\wx(S)}{\sum_{\ordered{S'} \in (V)_{k-1}} \wx(S')} \le \frac{n \cdot \mu^{-1} n^{-k+1}}{(k-1)! \cdot n} = O_{\mu}(n^{-k+1})\, ,
    \]
    where we used~\eqref{eq:RW_denominator}.
    By using in addition that $\delta(G) \ge (1/2+\eps)n$, we can show that the transition probabilities of $\cY$ are $O_{\mu}(n^{-1})$.
    Indeed, with $\ordered{Z_i}=(Y_{i-(k-2)},\dots,Y_i)$ being the ordered set of the last $k-1$ vertices of $\cY$, for $v \in V(G) \setminus Z_i$ we have
    \[    
        \Pr[Y_{i+1}=v|Y_{i-(k-2)},\dots,Y_{i}] =  \frac{\wx(Z_{i}\cup \Set{v})}{\wx(Z_i)} \le \frac{\mu^{-1} n^{-k+1}}{1/2 \cdot n \cdot \mu n^{-k+1}} = O_{\mu}(n^{-1})\, , 
    \]
    while for $v \in Z_i$ we have $\Pr[Y_{i+1}=v|Y_{i-(k-2)},\dots,Y_{i}]=0$.
    Therefore, by applying the chain rule, $\Pr[Y_1=q_1,\dots,Y_t=q_t]=O_{t,\mu}(n^{-t})$. 
    Moreover, the number of walks of order $t$ which are not self-avoiding is $O_t(n^{t-1})$ and thus $\Pr[\cB^c]=O_t(n^{t-1}) \cdot O_{t,\mu}(n^{-t})=O_{t,\mu}(n^{-1})$.

    Now Item~\ref{eq:sampling_1} follows easily as we have 
    $$\PP[P=Q]\le \frac{\Pr[Y_1=q_1,\dots,Y_t=q_t]+\Pr[Y_1=q_t,\dots,Y_t=q_1]}{\Pr[\cB]} =O_{t,\mu}(n^{-t}),
    $$  where we used that $\Pr[\cB] = 1-O_{t,\mu}(n^{-1})\ge 1/2$.

    Fix $v \in V(G)$ and $1 \le i \le t$.
    The number of walks of order $t$ which are not self-avoiding and whose $i$-th vertex is $v$ is $O_t(n^{t-2})$ and thus $\Pr\left[\{Y_i=v\} \cap \cB^c\right] = 
    O_t(n^{t-2}) \cdot O_{t,\mu}(n^{-t})=O_{t,\mu}(n^{-2})$.
    Therefore,
    \[
        \PP[v \in V(P)] = \Pr\left[\bigcup_{i \in [t]} \{Y_i=v\} \; | \; \cB \right] 
        = \sum_{i \in [t]} \frac{\Pr\big[\{Y_i=v\} \cap \cB \big]}{\Pr[\cB]} \ge tn^{-1}-O_{t,\mu}(n^{-2}) \, ,
    \]
    where we used that the events $\{Y_i=v\} \cap \cB$ with $i \in [t]$ are pairwise disjoint, that $\Pr[\{Y_i=v\}]=n^{-1}$ by Fact~\ref{fact:RW_uniform_vertices}, that 
    $$
    \Pr\big[\{Y_i=v\} \cap \cB \big] = \Pr\big[Y_i=v\big] - \Pr\big[\{Y_i=v\} \cap \cB^c\big] = n^{-1} - O_{t,\mu}(n^{-2}),
    $$
    and that $\Pr[\cB]\le 1$. 
    Similarly, by using $\Pr[\cB]=1-O_{t,\mu}(n^{-1})$, we get that $\PP[v \in V(P)] \leq tn^{-1}+O_{t,\mu}(n^{-2})$.
    This proves Item~\ref{eq:sampling_2}.

    For a given $k \le i \le t$, denote by $e_i:=\{Y_{i-k+1},\dots,Y_i\}$ the $(i-k+1)$-st edge of~$\cY$.
    Let $e \in E(G)$. 
    Similarly as above, $\Pr[\{e_i=e\} \cap \cB^c]=O_{t,\mu}(n^{-k-1})$ because there are $O_t(n^{t-k-1})$ tight walks which are not self-avoiding and satisfy $e_i = e$, and each such walk has probability $O_{t,\mu}(n^{-t})$. Thus, 
    \[
        \Pr\Big[e_i=e \; | \; \cB\Big]=\frac{\Pr\big[e_i=e\big]-\Pr\big[\{e_i=e\} \cap \cB^c\big]}{\Pr[\cB]} \ge \frac{k}{n}\cdot \wx(e) - O_{t,\mu}(n^{-k-1})\, ,
    \]
    where we used Fact~\ref{fact:RW_uniform_vertices} and $\Pr[\cB]\le 1$.   
    Let $\cR$ be the set of edges of $G$ which are coloured red.
    Then 
    \[
        \Pr\Big[e_i\in \cR \; | \; \cB\Big] \ge \sum_{e \in \cR} \left(\frac{k}{n}\cdot \wx(e) - O_{t,\mu}(n^{-k-1})\right) \ge \frac{1+\mu}{r}-O_{t,\mu}(n^{-1})\, ,
    \]
    where we used that $|\cR| \le n^k$ and $\sum_{e \in \cR} \wx(e) \ge (1+\mu)\frac{n}{rk}$.
    Since $P$ is a path of order $t$, it has $t-k+1$ edges, and thus Item~\ref{eq:sampling_3} follows by linearity of expectation. 
\endproof

\section{Finding a linear forest with high discrepancy}\label{sec:linear_forest}

The goal of this subsection is to prove the following result.

\begin{lemma}
\label{lem:path_forest}
    Let $1/n \ll 1/t \ll \beta \ll \mu \ll \eps, 1/k, 1/r$ with $k,r,t\in \bN$, $k\ge 3$, $r\ge 2$.
    Let $G$ be an $n$-vertex $k$-graph with $\delta(G)\ge (1/2+\eps)n$ whose edges are $r$-coloured. Then $G$ contains a collection of vertex-disjoint tight paths of order $t$ such that their union covers all but at most $\beta n$ vertices of $G$ and there is some colour which appears on at least $(1+\mu)\frac{n}{r}$ edges in the paths.
\end{lemma}

\proof
Let
\[
	1/n \ll \delta \ll 1/t \ll \beta \ll \mu \ll \eps, 1/k, 1/r\, ,
\]
where $2\mu$ is given by Lemma~\ref{lem:sampling} on input $\eps, r, k$, and set $\gamma:=\delta/(50t^2)$ and $V:=V(G)$. 

Set $N:=n^{t-1/2}$ and let $\cP:=\{P_i:i \in [N]\}$ be a collection of $N$ tight paths of order $t$ independently sampled according to the distribution given by Lemma~\ref{lem:sampling}.
For a given $v \in V$, define $X_v:=\{i \in [N]:v \in V(P_i)\}$.
Using
Item~\ref{eq:sampling_2} in Lemma~\ref{lem:sampling},
we have $\EE[|X_v|]=N \cdot t n^{-1} \cdot \left[1 \pm O_{t,\mu}(n^{-1})\right] = t n^{t-3/2} \cdot \left[1 \pm O_{t,\mu}(n^{-1})\right]$.
Therefore, by Chernoff's inequality (Lemma~\ref{lem:chernoff}) and a union bound over $v \in V$, w.h.p. we have 
\begin{equation}
\label{eq:X_v}
    |X_v| = (1 \pm \beta/3) t n^{t-3/2}
\end{equation}
for every $v \in V$.
For a path $P \in \mathcal{P}$, let $\text{red}(P)$ be the number of red edges of $P$, and note that $0 \leq \text{red}(P) \leq t-k+1$. Let 
$R:= \sum_{i = 1}^N \text{red}(P_i)$.
By~Item~\ref{eq:sampling_3} of Lemma~\ref{lem:sampling}, $\EE[R] \ge N \cdot \left[(t-k+1) \cdot \frac{1+2\mu}{r} - O_{t,\mu}(n^{-1})\right]$ and by Chernoff's inequality (Lemma~\ref{lem:chernoff}), we have w.h.p. that 
\begin{equation}
\label{eq:red}
    R \ge (1-\beta) \cdot (t-k+1) \cdot \frac{1+2\mu}{r} \cdot n^{t-1/2}\, .
\end{equation}

We now show that we can pass to a large subcollection $\cP' \subseteq \cP$ such that no two paths in $\cP'$ have the same vertex set.
Let $Y$ be the set of pairs $1 \leq i<j \leq N$ such that $V(P_i) = V(P_j)$. We now bound~$|Y|$. 
Given any $t$-subset $S$ of $V$, observe that there are $t!/2$ (unordered) paths $P$ with $V(P)=S$ and that, using Item~\ref{eq:sampling_1} in Lemma~\ref{lem:sampling}, the probability that $P_i = P$ for a given $1 \leq i \leq N$ is $O_{t,\mu}(n^{-t})$.
Therefore, the probability that a given pair $i,j$ belongs to $Y$ is $O_{t,\mu}(n^{-t})$, and thus $\EE[|Y|] \le N^2 \cdot O_{t,\mu}(n^{-t}) = O_{t,\mu}(n^{t-1})$.
It follows from Markov's inequality that w.h.p.~$|Y| \le n^{t-2/3}$.

We now fix such a collection of paths that satisfies~\eqref{eq:X_v},~\eqref{eq:red} and $|Y| \le n^{t-2/3}$.
Let $\mathcal{P}' \subseteq \mathcal{P}$ be obtained from $\mathcal{P}$ by deleting $P_i,P_j$ for every pair $\{i,j\} \in Y$.
Then $|\cP'| \ge |\cP| - 2|Y| \ge (1-\beta/3) n^{t-1/2}$.
Let $\cH$ be the auxiliary $t$-graph on $V$ with edge set $\{V(P):P \in \cP'\}$. 
By the definition of $\cP'$, we have $V(P) \neq V(Q)$ for each $P, Q \in \cP'$, and thus $e(\cH)=|\cP'|$ (i.e., $\mathcal{H}$ has no multiple edges).
We now show that $\cH$ is suitable for an application of Theorem~\ref{thm:EGJ}, by establishing bounds on $\Delta(\cH)$ and $\Delta^c(\cH)$.
Using~\eqref{eq:X_v}, we have $d_{\cH}(\{v\}) \le |X_v| \le (1+\beta/3)t n^{t-3/2}$.
Setting $\Delta:=(1 + \beta/3) t n^{t-3/2}$, we have $\Delta(\cH) \le \Delta$, and it trivially holds that $\Delta^c(\cH) \le n^{t-2} \le \Delta^{1-\delta}$, as $\delta \ll 1/t$.

We would like the matching given by Theorem~\ref{thm:EGJ} to be almost-spanning and with large discrepancy. 
To this end, we define two weight functions $w_1,w_2\colon E(\cH) \rightarrow \mathbb{R}_{\ge 0}$ as follows:
$w_1 \equiv 1$; and for an edge $e \in E(\cH)$ corresponding to a path $P \in \cP'$, we define $w_2(e) := \text{red}(P)$. 
We claim that $w_i(E(\cH)) \ge \max_{e \in E(\cH)} w(e) \Delta^{1+\delta}$ for each $i=1,2$.
For $i=1$, this is obvious as
\[
	w_1(E(\cH)) = e(\cH) \ge (1-\beta/3) n^{t-1/2} \ge \Delta^{1+\delta} = \max_{e \in E(\cH)} w_1(e) \Delta^{1+\delta}\, ,
\]
where the last inequality uses that $\delta \ll 1/t$.
And for $i=2$,  
we have
\begin{align*}
w_2(E(\cH)) &= \sum_{P \in \mathcal{P}'}\text{red}(P) = 
R - \sum_{P \in \mathcal{P}\setminus \mathcal{P}'}\text{red}(P) \geq R - (t-k+1) \cdot 2n^{t-2/3} \\ &\geq (1-2\beta) \cdot (t-k+1) \cdot \frac{1+2\mu}{r} \cdot n^{t-1/2},
\end{align*}
where the first inequality uses $|\cP|-|\cP'| \le 2 |Y| \le 2n^{t-2/3}$ and that each path has at most $t-k+1$ red edges, while the second inequality uses~\eqref{eq:red}.
Therefore
\[
	w_2(E(\cH)) = \Omega(n^{t-1/2}) \ge (t-k+1) \Delta^{1+\delta} \ge  \max_{e \in E(\cH)} w_2(e) \Delta^{1+\delta}\, ,
\]
using that $\Delta = O_t(n^{t-3/2})$ and $\delta \ll 1/t$.

Let $\cM$ be the matching in $\cH$ given by applying Theorem~\ref{thm:EGJ} with 
$\cW := \{w_1,w_2\}$. Using $w_1(E(\cH)) = e(\mathcal{H}) = |\cP'| \ge (1-\beta/3)n^{t-1/2}$ and the guarantees of Theorem~\ref{thm:EGJ}, we have
\[
	|\cM|=w_1(\cM) \ge \left(1-\Delta^{-\gamma}\right) \cdot \frac{w_1(E(\cH))}{\Delta} \ge 
    \left(1-\Delta^{-\gamma}\right) \cdot \frac{(1-\beta/3)n^{t-1/2}}{(1 + \beta/3) t n^{t-3/2}} \geq 
    (1-\beta) \frac{n}{t}\,.
\]
Similarly, for $w_2$ we have
\begin{align*}
	w_2(\cM) \ge \left(1-\Delta^{-\gamma}\right) \cdot \frac{w_2(E(\cH))}{\Delta} & \ge (1-o(1)) \cdot \frac{(1-2\beta) \cdot (t-k+1) \cdot \frac{1+2\mu}{r} \cdot n^{t-1/2}}{(1 + \beta/3) t n^{t-3/2}} \\ &= 
    (1 - o(1)) \cdot 
    \frac{1-2\beta}{1+\beta/3} \cdot \frac{t-k+1}{t} \cdot \frac{1+2\mu}{r} \cdot n \\
    &\ge (1-3\beta) \cdot \frac{1+2\mu}{r} \cdot n \ge (1+\mu)\frac{n}{r}\, ,
\end{align*}
where we used the bound on $w_2(E(\cH))$ established above, the definition of $\Delta$, together with $1/t \ll \beta \ll \mu \ll 1/k$. 

Therefore the matching $\cM$ corresponds to a collection of vertex-disjoint paths of $G$ of order $t$ such that their union covers all but at most $\beta n$ vertices and the colour red appears on at least $(1+\mu)\frac{n}{r}$ edges in the paths, as desired.
\endproof

\section{Proof of the main theorem}\label{sec:main proof}

We are now ready to prove our main result. 

\begin{proof}[Proof of Theorem~\ref{thm:main}]

Let 
\[
    1/n \ll 1/t \ll \beta \ll \mu \ll \eps \ll 1/k, 1/r\,,
\]
where we have assumed without loss of generality that $\eps$ is sufficiently small.
Let $G$ be an $n$-vertex $k$-graph with $\delta(G) \ge (1/2+\eps)n$ whose edges are $r$-coloured.
By Lemma~\ref{lem:absorber}, there exists a tight path $P_0$ such that $|V(P_0)| \le \mu n$ and for each $W \subseteq V(G) \setminus V(P_0)$ of size $|W| \le 3 \beta n$ there is a tight path covering $V(P_0) \cup W$ and with the same ends as $P_0$.
Let $R$ be a random subset of $V(G) \setminus V(P_0)$ obtained by including each vertex with probability $\beta$ independently. Then w.h.p.~it holds that $\beta n/2 \leq |R| \leq 2\beta n$ and 
that $|N_G(S) \cap R| \ge (1/2+\eps/2)|R|$ for every $(k-1)$-subset $S \subseteq V(G)$. Indeed, this can be shown by a standard application of Chernoff's inequality (Lemma~\ref{lem:chernoff}) and a union bound over~$S$.

Let $G':=G \setminus (V(P_0) \cup R)$ and observe that $\delta(G') \ge \delta(G) - |V(P_0) \cup R| \ge (1/2+\eps/2)n \ge (1/2+\eps/2)|V(G')|$.
Then, by Lemma~\ref{lem:path_forest} (with $4\mu$ playing the role of $\mu$), $G'$ contains a collection $\cP:=\{P_i:i \in [N]\}$ of vertex-disjoint tight paths of order $t$ such that their union covers all but at most $\beta n$ vertices, and there is a colour, say red, which appears on at least $(1+4\mu)|V(G')|/r \ge (1+4\mu)(1-\mu-2\beta)n/r \ge (1+\mu)n/r$ edges in the paths.
Therefore, if we manage to connect the paths in $\cP$ into a tight cycle while covering all the vertices then we are done. This connection can be achieved using the standard tools collected in Section~\ref{sec:dirac_graphs}. The details follow. 

Denote by $\ordered{S_i}$ and $\ordered{T_i}$ the ends of $P_i$ for each $0 \le i \le N$.
Add the uncovered vertices of $G'$ to $R$ to get a set $R'$, and observe that 
$|R| \leq |R'| \leq |R| + \beta n$. 
Using multiple applications of the connecting lemma (Lemma~\ref{lem:path connect}), we can connect the path $P_0$ and the paths in $\cP$ into an almost-spanning tight cycle using vertices in $R'$, as proved by the following claim.
\begin{claim}
	For each $0 \le i \le N$, there is a tight path $Q_i$ of order at most $2k/\eps^2$ connecting $\overleftarrow{T_i}$ and $\overleftarrow{S_{i+1}}$ (where indices are modulo $N+1$), 
    such that $V(Q_i) \setminus (T_i \cup S_{i+1}) \subseteq R'$.
	Moreover, we can choose such paths to be pairwise vertex-disjoint. 
\end{claim}
\begin{proof}
	Suppose we can find vertex-disjoint connecting paths $Q_0, \dots, Q_{m-1}$ as in the statement of the claim, and let $m$ be as large as possible.
	If we are not done yet, then $m \le N$. The union of $Q_0, \dots, Q_{m-1}$ covers at most $m \cdot 2k/\eps^2 \le 2kn/(\eps^2 t)$ vertices of $R'$, where we used that $N \le n/t$.
	Let $R''$ denote the subset of $R'$ consisting of the uncovered vertices.
	Then for every $(k-1)$-subset $S \subseteq V(G)$ we have 
    $|N_G(S) \cap R''| \ge 
    (1/2 + \eps/2)|R| - |R' \setminus R''| \geq 
    (1/2+\eps/4)|R'' \cup T_{m+1} \cup S_{m+2}|$, where we used that $|R| \geq |R'| - \beta n$, $|R' \setminus R''| \leq 2kn/(\eps^2 t)$ and
    $1/t \ll \beta \ll \eps, 1/k$.
	Therefore, we can apply Lemma~\ref{lem:path connect} to $G[R'' \cup T_{m+1} \cup S_{m+2}]$ to get a tight path of order at most $2k/\eps^2$ connecting $\overleftarrow{T_{m+1}}$ and $\overleftarrow{S_{m+2}}$, which is vertex-disjoint from $Q_0, \dots, Q_{m-1}$.
	This contradicts the maximality of $m$.
\end{proof}
	Observe that $C:=\bigcup_{0 \le i \le N} (P_i \cup Q_i)$ is a tight cycle. Now let $W \subseteq R'$ be the subset of vertices not covered by $C$ and observe that clearly $|W| \le |R'| \le 3\beta n$.
    By the property of the absorbing path $P_0$, there exists a tight path $\tilde{P_0}$ which covers $V(P_0) \cup W$ and has the same ends as $P_0$, i.e. $\overrightarrow{S_0}$ and $\overrightarrow{T_0}$.
    It follows that $\tilde{C}:=\tilde{P_0} \cup \bigcup_{1 \le i \le N} (P_i \cup Q_i)$ is a tight Hamilton cycle of $G$ and, since each edge of $P_1 \cup \dots \cup P_N$ is an edge of $\tilde{C}$ as well, then $\tilde{C}$ has at least $(1+\mu)n/r$ red edges, as desired.
\end{proof}

\section{Concluding remarks}

The main result of this paper offers a discrepancy version of the celebrated result of R\"odl, Ruci\'nski and Szemer\'edi~\cite{RRS:08}, and determines the minimum $(k-1)$-degree threshold for high discrepancy of tight Hamilton cycles and perfect matchings; both of these thresholds equal $1/2$, which is also the existence threshold for these structures. 
In the following we discuss some natural open problems for further research.

\begin{itemize}
\item[\textbullet]
A very natural question is to study minimum-degree discrepancy thresholds for the $j$-degree\footnote{The \defn{minimum $j$-degree} of a hypergraph is the minimum of $d(S)$ over all sets $S$ of $j$ vertices.} with $j < k-1$. 
We remark that for this question, the existence threshold for perfect matchings (i.e.~the minimum $j$-degree guaranteeing the existence of a perfect matching) is mostly not known.
We wonder if, for $k$-uniform hypergraphs with $k \geq 3$, there is some $j$ such that the $j$-degree discrepancy threshold is strictly larger than the corresponding existence threshold (as is the case for graphs). 

\textbf{Remark added.} This problem has been solved for each $j \neq 1$ in a simultaneous work of Balogh, Treglown and Z\'arate-Guer\'en~\cite{BTZ-G:24}, who also provided a construction showing that, for $j=1$ and $k=3$, the discrepancy threshold is significantly larger than the existence threshold. 
The case $j=1$ was then fully resolved by Lu, Ma and Xie~\cite{LMX24+}, and, independently, by H{\`a}n, Lang, Marciano, Pavez-Sign{\'e}, Sanhueza-Matamala, Treglown and Z{\'a}rate-Guer{\'e}n~\cite{HLMPSTZ24+}.

\item[\textbullet] Another natural question is to consider other notions of Hamilton cycles. As mentioned in the introduction, Theorem~\ref{thm:main} implies that minimum $(k-1)$-degree $(1/2 + \varepsilon)n$ guarantees the existence of Hamilton $\ell$-cycles of high discrepancy, for every $1 \leq \ell \leq k-1$. However, unlike in the case of tight Hamilton cycles (namely, $\ell = k-1$), we do not have a matching lower bound. For example, it is known that the minimum $(k-1)$-degree threshold for the existence of loose Hamilton cycles in $k$-graphs is $\frac{1}{2(k-1)}$, see~\cite{KKMO:11,KO:06}. We wonder if this is also the discrepancy threshold of loose Hamilton cycles.  

\item[\textbullet] As mentioned in the introduction, $1/2$ is the minimum $(k-1)$-degree threshold for the existence of perfect matchings in $k$-graphs (see~\cite{KO:06matchings,RRS:09}), which shows that Corollary~\ref{cor:PM} is tight. However, it is also known~\cite{KO:06matchings,RRS:09} that a $k$-graph with minimum $(k-1)$-degree at least $(1+o(1))\frac{n}{k}$ contains a {\em near-perfect matching}, i.e.~a matching of size $\frac{n}{k} - O_k(1)$. 
For $k=3$ and $2$ colours, we have the following simple example showing that $1/2$ is the discrepancy threshold of near-perfect matchings (which is larger than the existence threshold of $1/3$). Partition the vertices into two sets $A$ and $B$ of equal size and take $G$ to be the hypergraph consisting of all edges which intersect both $A$ and~$B$. Colour in red the edges which intersect $A$ in two vertices, and colour the remaining edges in blue. Every matching of size $n/3 - t$ must have at least $n/6-2t$ edges in each colour, meaning that there is no near-perfect matching with discrepancy $\Omega(n)$. 
It is therefore natural to ask, for general $k \geq 3$ and $r \geq 2$, what is the minimum $(k-1)$-degree threshold guaranteeing a near-perfect matching of high discrepancy in every $r$-edge-colouring.

\item[\textbullet]
It would also be interesting to prove similar results for other spanning structures in hypergraphs, perhaps even of design-type, such as Steiner triple systems.
We will return to this in a future work.

\item[\textbullet]
Instead of studying minimum degree thresholds, one might also consider random hypergraphs. In the graph case, Gishboliner, Krivelevich and Michaeli~\cite{GKM_Hamilton} showed that with high probability, the random graph $G(n,p)$ has the following property: in every $r$-edge-colouring, there exists a Hamilton cycle which has at least roughly $\frac{2n}{r+1}$ edges of the same colour and hence a perfect matching with at least roughly $\frac{n}{r+1}$ edges of the same colour.  
The respective constants $\frac{2}{r+1}$ and $\frac{1}{r+1}$ are best possible even in the complete graph. This raises the question of whether the same phenomenon holds in hypergraphs. For example, as mentioned in the introduction, the result of Alon--Frankl--Lov\'asz~\cite{AFL:86} implies that every $2$-edge-colouring of $K_n^{(k)}$ has a perfect matching with at least roughly $\frac{n}{k+1}$ edges of the same colour. Is the same true in a random $k$-graph (say, with edge probability above the existence threshold $(\log n)/n^{k-1}$)? 

\item[\textbullet]
Generalizing the previous item, it would be very interesting to prove a general result relating the threshold for containing a structure (in a random graph/hypergraph) to the threshold for having high discrepancy for this structure. Namely, for a family $\mathcal{F}$ of graphs (or hypergraphs) on $[n]$, let $p_0$ be the threshold for the event that $G \sim G(n,p)$ contains a member from $\mathcal{F}$, and let $p_1$ be the threshold for the event that in every $2$-edge-colouring of $G \sim G(n,p)$, there is a member $F \in \mathcal{F}$ with high discrepancy, say of order $\Theta(e(F))$. 
Note that $p_1$ is well-defined if (and only if) $\mathcal{F}$ has high discrepancy in $K_n$. 
Clearly $p_1 \geq p_0$. Is there a general upper bound on $p_1$ in terms of $p_0$? We wonder if the recent breakthroughs around the expectation-threshold conjecture are relevant to this question.

\item[\textbullet]
Even more generally, one can ask about the discrepancy of random subhypergraphs of general hypergraphs (not necessarily those arising from graphs). 
Namely, we return to the original definition of discrepancy, where $\cH$ is a hypergraph, and we colour the vertices of $\cH$ with two colours. How ``robust'' is discrepancy? For instance, suppose $\cH$ has high discrepancy, and we take a random subset of vertices by including each vertex independently with probability~$p$. Is the random induced subhypergraph likely to still have high discrepancy?
\end{itemize}

\section*{Acknowledgements}

We thank the anonymous referees for their valuable comments.
This research was initiated while AS was visiting the University of Passau; he would like to thank the University of Passau for the hospitality and the stimulating research environment.
\bibliographystyle{amsplain_v2.0customized}
\bibliography{References}

\end{document}